\documentclass[a4paper,11pt]{amsart}
\usepackage{amsmath,latexsym,amssymb,amsfonts}
\usepackage{amscd,amssymb,epsfig}
\usepackage[arrow, matrix, curve]{xy}
\usepackage{hyperref}
\usepackage{mathdots}
\usepackage{etoolbox} 

\renewcommand\labelenumi{(\roman{enumi})}
\renewcommand\theenumi\labelenumi

\numberwithin{equation}{section}

\newtheorem{theorem}{Theorem}[section]
\newtheorem{lemma}[theorem]{Lemma}
\newtheorem{corollary}[theorem]{Corollary}
\newtheorem{proposition}[theorem]{Proposition}

\theoremstyle{definition}

\newtheorem{example}[theorem]{Example}
\newtheorem{remark}[theorem]{Remark}
\AtEndEnvironment{remark}{~\hfill$\blacksquare$}

\renewcommand{\epsilon}{\varepsilon}

\newcommand{\A}{\mathcal{A}}

\newcommand{\C}{\mathbb{C}}

\newcommand{\bE}{\mathbb{E}}
\newcommand{\F}{\mathcal{F}}
\newcommand{\bG}{\mathbf{G}}

\newcommand{\bH}{\mathbb{H}}
\newcommand{\cL}{\mathcal{L}}
\newcommand{\M}{\mathcal{M}}
\newcommand{\N}{\mathbb{N}}

\renewcommand{\P}{\mathcal{P}}
\newcommand{\bP}{\mathbb{P}}
\newcommand{\R}{\mathbb{R}}

\newcommand{\bS}{\mathbb{S}}

\newcommand{\id}{\operatorname{id}}

\newcommand{\ev}{\operatorname{ev}}
\newcommand{\tr}{\operatorname{tr}}
\newcommand{\Tr}{\operatorname{Tr}}

\newcommand{\vN}{\operatorname{vN}}

\newcommand{\sa}{{\operatorname{sa}}}

\newcommand{\supp}{\operatorname{supp}}

\renewcommand{\1}{\operatorname{\textbf{1}}}

\renewcommand{\star}{\bigstar}
\newcommand{\ootimes}{\mathbin{\overline{\otimes}}}

\renewcommand{\Re}{\operatorname{Re}}
\renewcommand{\Im}{\operatorname{Im}}

\makeatletter
\def\moverlay{\mathpalette\mov@rlay}
\def\mov@rlay#1#2{\leavevmode\vtop{%
\baselineskip\z@skip \lineskiplimit-\maxdimen
\ialign{\hfil$#1##$\hfil\cr#2\crcr}}}
\makeatother

\makeatletter
\def\@settitle{\begin{center}%
  \baselineskip14\p@\relax
    \normalfont \large \uppercase{\textbf{\@title}}
  \end{center}%
}

\makeatother

\makeindex

\title[H\"older Continuity for Noncommutative Polynomials]{H\"older Continuity of Cumulative Distribution Functions for Noncommutative Polynomials under Finite Free Fisher Information}

\author[M. Banna]{Marwa Banna}
\address{Saarland University, Faculty of Mathematics, D-66123 Saarbr\"ucken, Germany}
\email{banna@math.uni-sb.de}

\author[T. Mai]{Tobias Mai}
\address{Saarland University, Faculty of Mathematics, D-66123 Saarbr\"ucken, Germany}
\email{mai@math.uni-sb.de}

\date{\today}

\thanks{This work has been supported by the ERC Advanced Grant NCDFP 339760 held by Roland Speicher.}

\keywords{free Fisher information and entropy, noncommutative polynomials, H\"older continuity, Kolmogorov distance, random matrices, Gibbs laws. }

\subjclass[2000]{46L54, 60B10, 60B20}

\begin{document}

\begin{abstract}
This paper contributes to the current studies on regularity properties of noncommutative distributions in free probability theory. More precisely, we consider evaluations of selfadjoint noncommutative polynomials in noncommutative random variables that have finite non-microstates free Fisher information, highlighting the special case of Lipschitz conjugate variables. For the first time in this generality, it is shown that the analytic distributions of those evaluations have H\"older continuous cumulative distribution functions with an explicit H\"older exponent that depends only on the degree of the considered polynomial. For linear polynomials, we reach in the case of finite non-microstates free Fisher information the optimal H\"older exponent $\frac{2}{3}$, and get Lipschitz continuity in the case of Lipschitz conjugate variables. In particular, our results guarantee that such polynomial evaluations have finite logarithmic energy and thus finite (non-microstates) free entropy, which partially settles a conjecture of Charlesworth and Shlyakhtenko \cite{CS16}.

We further provide a very general criterion that gives for weak approximations of measures having H\"older continuous cumulative distribution functions explicit rates of convergence in terms of the Kolmogorov distance.

Finally, we combine these results to study the asymptotic eigenvalue distributions of polynomials in GUEs or matrices with more general Gibbs laws. For Gibbs laws, this extends the corresponding result obtained in \cite{GS09} from convergence in distribution to convergence in Kolmogorov distance; in the GUE case, we even provide explicit rates, which quantify results of \cite{HT05,HST06} in terms of the Kolmogorov distance.
\end{abstract}

\maketitle

\section{Introduction}

Noncommutative distributions are at the heart of noncommutative probability theory and of free probability theory in particular. These, in general, purely combinatorial objects allow some very elegant translation of various questions arising for instance in operator algebra or random matrix theory into the unifying language of noncommutative probability theory; in this way, they can build bridges between originally unrelated fields and often also make available tools from free probability theory in those areas.

Within the algebraic frame of a \emph{noncommutative probability space} $(\A,\phi)$, i.e., a unital complex algebra $\A$ with a distinguished unital linear functional $\phi:\A\to\C$, the \emph{(joint) noncommutative distribution} of a tuple $X=(X_1,\dots,X_n)$ consisting of finitely many noncommutative random variables $X_1,\dots,X_n\in\A$ is given as the linear functional
$$\mu_X:\ \C\langle x_1,\dots,x_n\rangle \to \C,\qquad P \mapsto \phi(P(X)).$$
It is defined on the algebra $\C\langle x_1,\dots,x_n\rangle$ of noncommutative polynomials in $n$ formal noncommuting variables $x_1,\dots,x_n$.

In practice, one often works -- as we will do in the following -- in the more analytic setting of a \emph{tracial $W^\ast$-probability space} $(\M,\tau)$, i.e., a von Neumann algebra $\M$ that is endowed with some faithful normal tracial state $\tau: \M\to\C$.
If tuples $X=(X_1,\dots,X_n)$ of noncommutative random variables $X_1,\dots,X_n$ in $\M$ are considered, then their joint noncommutative distribution $\mu_X$ determines the generated von Neumann algebra $\vN(X_1,\dots,X_n)$ up to isomorphism. Thus, $\mu_X$ provides a kind of combinatorial ``barcode'' for $\vN(X_1,\dots,X_n)$ and consequently contains all spectral properties of $X_1,\dots,X_n$; however, the challenging question is how to read off those information from a given $\mu_X$.

Here, we are concerned with regularity properties of noncommutative distributions. 

In a series of papers \cite{Voi93,Voi94,Voi96,Voi97,Voi98,Voi99}, Voiculescu developed free probability analogues of the classical notions of Fisher information and entropy; see \cite{Voi02} for a survey. Here, we follow the non-microstates approach that Voiculescu presented in \cite{Voi98,Voi99}. To tuples $X=(X_1,\dots,X_n)$ of noncommutative random variables in a tracial $W^\ast$-probability space $(\M,\tau)$, he associates the \emph{non-microstates free Fisher information $\Phi^\ast(X)$} and the \emph{non-microstates free entropy $\chi^\ast(X)$}. Each of those numerical quantities, if finite, gives some rich structure to the joint noncommutative distribution $\mu_X$, however, without determining it completely.
Into that context fits also the notion of \emph{Lipschitz conjugate variables}, which provides a strengthening of $\Phi^\ast(X)<\infty$; it was introduced in \cite{Dab14} and taken up again in \cite{DI16}.
While it is the common viewpoint that each of those conditions entails some strong regularity of $\mu_X$, making this guess precise remains quite intricate.

One of the major drawbacks in that respect is the lack of an effective analytic machinery to handle noncommutative distributions in a way which is similar to the measure theoretic description of distributions in classical probability theory.

Such tools are available only in very limited situations. Even in the strong analytic framework of a tracial $W^\ast$-probability space $(\M,\tau)$, we typically must restrict ourselves to the case of a single noncommutative random variable $X\in\M$ in order to gain such an analytic description. For instance, if the considered operator $X$ is selfadjoint, then its combinatorial noncommutative distribution can be encoded by some compactly supported Borel probability measure $\mu_X$ on the real line $\R$, called the \emph{analytic distribution of $X$}; more precisely, the analytic distribution $\mu_X$ is uniquely determined among all Borel measures on $\R$ by the requirement that
$$\tau(X^k) = \int_\R t^k\, d\mu_X(t) \qquad\text{for all integers $k\geq 0$}.$$
For the sake of completeness, we note that this notion can be generalized to normal operators $X$, resulting in a compactly supported Borel probability measure on the complex plane $\C$; on the other hand, for operators that fail to be normal, on can study instead its so-called Brown measure.

Accordingly, it is not even clear what ``regularity'' should mean for general noncommutative distributions.

In recent years, evaluations of ``noncommutative test functions'' such as noncommutative polynomials or noncommutative rational functions were successfully developed as a kind of substitute for the measure theoretic description in order to overcome those difficulties. In fact, each such evaluation produces a single noncommutative random variable whose analytic distribution can be studied by measure theoretic means. The guiding idea is that the larger the considered class of test functions is, the more information one gains about the underlying multivariate noncommutative distribution.

In this way, also the aforementioned problem becomes treatable: ``regularity'' of noncommutative distributions $\mu_X$, imposed by conditions such as $\Phi^\ast(X) < \infty$ and $\chi^\ast(X) > - \infty$, is understood as being reflected in properties of the analytic distributions $\mu_{f(X)}$ that arise from evaluations $f(X)$ of noncommutative test functions $f$.

Several results have already been obtained in that direction; see, for instance, \cite{SS15,CS16,MSW17,MSY18}. We elaborate here on the H\"older continuity of cumulative distribution functions of analytic distributions associated to noncommutative polynomial evaluations in variables having finite Fisher information.

Recall that the \emph{cumulative distribution function $\F_\mu$} of a probability measure $\mu$ on $\R$ is the function $\F_\mu: \R \to [0,1]$ that is defined by $\F_\mu(t) := \mu((-\infty,t])$; if the analytic distribution $\mu_X$ of a single selfadjoint noncommutative random variable $X$ in $(\M,\tau)$ is considered, we will abbreviate $\F_{\mu_X}$ by $\F_X$.
We say that $\F_\mu$ is \emph{H\"older continuous with exponent $\beta\in(0,1]$} if there exists a constant $C>0$, called a \emph{H\"older coefficient of $\F_\mu$}, such that
\begin{equation}\label{eq:Hoelder_condition}
|\F_\mu(t)-\F_\mu(s)| \leq C |t-s|^{\beta} \qquad\text{for all $s,t\in\R$}.
\end{equation}

By our first main result, we establish for the first time H\"older continuity with explicit values for both the H\"older coefficients and the H\"older exponents for each non-constant selfadjoint noncommutative polynomial.

\begin{theorem}\label{thm:Hoelder_continuity}
Let $X_1,\dots,X_n$ be selfadjoint noncommutative random variables in some tracial $W^\ast$-probability space $(\M,\tau)$. Further, suppose that $P\in\C\langle x_1,\dots,x_n\rangle$ is a selfadjoint noncommutative polynomial of degree $d\geq 1$. Consider the associated selfadjoint noncommutative random variable $Y:=P(X_1,\dots,X_n)$ in $\M$. Then the following statements hold true for the cumulative distribution function $\F_Y$ of the analytic distribution $\mu_Y$ of $Y$:
\begin{enumerate}
 \item If $\Phi^\ast(X_1,\dots,X_n)<\infty$, i.e., if $(X_1,\dots,X_n)$ admits conjugate variables $(\xi_1,\dots,\xi_n)$, then $\F_Y$ is H\"older continuous with exponent $\frac{2}{3(2^d-1)}$.
 \item If $(X_1,\dots,X_n)$ admits Lipschitz conjugate variables $(\xi_1,\dots,\xi_n)$, then $\F_Y$ is H\"older continuous with exponent $\frac{1}{2^d-1}$.
\end{enumerate}
In each of those cases, \eqref{eq:Hoelder_constant-2a} and \eqref{eq:Hoelder_constant-2b}, respectively, provide explicit H\"older coefficients.
\end{theorem}

Theorem \ref{thm:Hoelder_continuity} has some important consequences.
It was shown in \cite{Jam15} that Borel probability measures with H\"older continuous cumulative distribution functions have finite logarithmic energy. If the analytic distribution $\mu_Y$ of a selfadjoint noncommutative random variable $Y\in\M$ is considered, then the latter quantity is known to be closely related to the non-microstates free entropy $\chi^\ast(Y)$, which coincides in that case with the microstates free entropy $\chi(Y)$; see \cite{Voi98}.
Thus, in summary, we obtain the following result.

\begin{theorem}\label{thm:finite_entropy}
Let $(\M,\tau)$ be a tracial $W^\ast$-probability space and suppose that $X_1,\dots,X_n$ are selfadjoint noncommutative random variables in $\M$ satisfying $\Phi^\ast(X_1,\dots,X_n)<\infty$. Then, for every selfadjoint noncommutative polynomial $P\in\C\langle x_1,\dots,x_n\rangle$ which is non-constant, we have that
$$\chi^\ast(P(X_1,\dots,X_n)) > -\infty.$$
\end{theorem}

This provides a partial and conceptual answer to a question formulated in \cite{CS16}. There, it is conjectured that the conclusion of Theorem \ref{thm:finite_entropy}, i.e., that $\chi^\ast(P(X))>-\infty$ holds for every non-constant noncommutative polynomial $P$, remains true under the weaker condition $\chi^\ast(X)>-\infty$ on $X=(X_1,\dots,X_n)$. 
At first sight, as we have strengthened that condition to $\Phi^\ast(X)<\infty$, it might be tempting to guess that this should even enforce $\Phi^\ast(P(X)) < \infty$. This guess, however, is much too optimistic, as one already sees in the case of a single variable: for a standard semicircular variable $S$, we have that $\Phi^\ast(S) < \infty$, in fact with Lipschitz conjugate variables, while $S^2$ is a free Poisson distribution, for which we know that $\Phi^\ast(S^2)=\infty$.
Thus, also under the stronger assumption that $X$ admits Lipschitz conjugate variables, one cannot hope in general for more than $\chi^\ast(P(X))>-\infty$.

For linear polynomials, it turns out that the H\"older exponents provided by Theorem \ref{thm:Hoelder_continuity} are optimal; see Example \ref{ex:optimality}. The precise statement reads as follows.

\begin{theorem}\label{thm:Lipschitz_absolute_continuity}
Let $(\M,\tau)$ be a tracial $W^\ast$-probability space and suppose that $X_1,\dots,X_n$ are selfadjoint noncommutative random variables in $\M$. Consider any operator $Y$ of the form
$$Y = a_0 + a_1 X_1 + \dots + a_n X_n$$
with $a_0\in\R$ and non-zero $(a_1,\dots,a_n) \in \R^n$. Then the following statements hold true for the cumulative distribution function $\F_Y$ of $Y$:
\begin{enumerate}
 \item If $\Phi^\ast(X_1,\dots,X_n)<\infty$, then $\F_Y$ is H\"older continuous with exponent $\frac{2}{3}$, where the value $\frac{2}{3}$ is optimal.
 \item If $(X_1,\dots,X_n)$ admits Lipschitz conjugate variables, then the cumulative distribution function $\F_Y$ is Lipschitz continuous on $\R$. In particular, the analytic distribution $\mu_Y$ of $Y$ is absolutely continuous with respect to the Lebesgue measure on $\R$ and has a bounded density.
\end{enumerate}
\end{theorem}

The rest of this paper is organized as follows.

In Section \ref{sec:free_differential_operators}, we recall some basic facts from the $L^2$-theory for free differential operators as initiated by Voiculescu.
The proof of Theorem \ref{thm:Hoelder_continuity} will be given in Section \ref{sec:Hoelder_continuity}; for that purpose, we will first collect and extend there some of the more recent results on which the proof builds.
In Section \ref{sec:finite_entropy}, we present the proof of Theorem \ref{thm:finite_entropy}.

Of independent interest is Section \ref{sec:Kolmogorov}, which is devoted to the phenomenon that convergence in distribution of Borel probability measures on $\R$ to a limit measure with H\"older continuous cumulative distribution function automatically improves itself to convergence in Kolmogorov distance. With Theorems \ref{thm:Hoelder_criterion} and \ref{thm:Hoelder_criterion_compact}, we prove quantified versions thereof that provide explicit rates of convergence for the Kolmogorov distance.

In the last Section \ref{sec:random_matrices}, we combine our previously obtained results and apply them to a wide class of random matrix models. More precisely, we consider tuples $(X_1^{(N)},\dots,X_n^{(N)})$ of $N\times N$ selfadjoint random matrices following some Gibbs law and whose asymptotic behavior as $N\to\infty$ is described by a tuple $(X_1,\dots,X_n)$ of selfadjoint noncommutative random variables with the property $\Phi^\ast(X_1,\dots,X_n) < \infty$. We then prove, in Corollaries \ref{cor:random_matrices_polynomial} and \ref{cor:random_matrices_block}, that the limiting eigenvalue distribution of a random matrix of the form $Y^{(N)}=f(X_1^{(N)},\dots,X_n^{(N)})$, for certain ``noncommutative functions'' $f$, has a H\"older continuous cumulative distribution function and that this convergence holds with respect to the Kolmogorov distance. Finally, we provide in Corollaries \ref{cor:block-GUE} and \ref{cor:p-GUE} rates of convergence of the Kolmogorov distance for the particular cases where $(X_1^{(N)},\dots,X_n^{(N)})$ is a tuple of independent GUE random matrices. The latter results rely on linearization techniques that are outlined in Appendix \ref{sec:linearization}.


\tableofcontents


\section{A glimpse on the $L^2$-theory for free differential operators}\label{sec:free_differential_operators}

This section is devoted to the $L^2$-theory for free differential operators, which underlies the non-microstates approach to free entropy as developed by Voiculescu in \cite{Voi98,Voi99}. For the reader's convenience, we recall here the needed terminology and some fundamental results.

\subsection{Noncommutative polynomials and noncommutative derivatives}

As usual, we will denote by $\C\langle x_1,\dots,x_n\rangle$ the unital complex algebra of \emph{noncommutative polynomials} in $n$ formal noncommuting variables $x_1,\dots,x_n$. Let us recall that any noncommutative polynomial $P\in \C\langle x_1,\dots,x_n\rangle$ can be written in the form
\begin{equation}\label{eq:ncpoly}
P = \sum^d_{k=0} \sum_{1\leq i_1,\dots,i_k\leq n} a_{i_1,\dots,i_k}\, x_{i_1} \cdots x_{i_k}.
\end{equation}
for some integer $d\geq 0$ and coefficients $a_{i_1,\dots,i_k} \in \C$; if there exist $1\leq i_1,\dots,i_d\leq n$ such that $a_{i_1,\dots,i_d} \neq 0$, then we say that \emph{$P$ has degree $d$} and we put $\deg(P) := d$.

Note that $\C\langle x_1,\dots,x_n\rangle$ becomes a $\ast$-algebra if it is endowed with the involution defined by
$$P^\ast = \sum^d_{k=0} \sum_{1\leq i_1,\dots,i_k\leq n} \overline{a_{i_1,\dots,i_k}}\, x_{i_k} \cdots x_{i_1}$$
for every noncommutative polynomial $P\in\C\langle x_1,\dots,x_n\rangle$ which is written in the form \eqref{eq:ncpoly}.

Elements in the algebraic tensor product $\C\langle x_1,\dots,x_n\rangle \otimes \C\langle x_1,\dots,x_n\rangle$ will be called \emph{bi-polynomials} in the following. Note that $\C\langle x_1,\dots,x_n\rangle \otimes \C\langle x_1,\dots,x_n\rangle$ forms by definition a unital complex algebra and moreover a $\C\langle x_1,\dots,x_n\rangle$-bimodule with the natural left and right action determined by $P_1 \cdot (Q_1 \otimes Q_2) \cdot P_2 := (P_1 Q_1) \otimes (Q_2 P_2)$. 
Therefore, we may introduce on $\C\langle x_1,\dots,x_n\rangle$ the so-called \emph{noncommutative derivatives} $\partial_1,\dots,\partial_n$ as the unique derivations
$$\partial_j:\ \C\langle x_1,\dots,x_n\rangle \to \C\langle x_1,\dots,x_n\rangle \otimes \C\langle x_1,\dots,x_n\rangle,\qquad j=1,\dots,n,$$
with values in $\C\langle x_1,\dots,x_n\rangle \otimes \C\langle x_1,\dots,x_n\rangle$ that satisfy $\partial_j x_i = \delta_{i,j} 1 \otimes 1$ for all $i,j=1,\dots,n$.

\subsection{Conjugate systems and non-microstates free Fisher information}

Let $(\M,\tau)$ be a tracial $W^\ast$-probability space (i.e., a von Neumann algebra $\M$ that is endowed with a faithful normal tracial state $\tau: \M \to \C$) and consider $n$ selfadjoint noncommutative random variables $X_1,\dots,X_n\in \M$.
Throughout the following, we will denote in such cases by $\M_0\subseteq \M$ the von Neumann subalgebra that is generated by $X_1,\dots,X_n$; in order to simplify the notation, the restriction of $\tau$ to $\M_0$ will be denoted again by $\tau$.

In \cite{Voi98}, Voiculescu associated to the tuple $(X_1,\dots,X_n)$ the so-called \emph{non-microstates free Fisher information $\Phi^\ast(X_1,\dots,X_n)$}; note that, while he assumed for technical reasons in addition that $X_1,\dots,X_n$ do not satisfy any non-trivial algebraic relation over $\C$, it was shown in \cite{MSW17} that this constraint is not needed as an a priori assumption on $(X_1,\dots,X_n)$ but is nonetheless enforced a posteriori by some general arguments. We call $(\xi_1,\dots,\xi_n) \in L^2(\M_0,\tau)^n$ a \emph{conjugate system for $(X_1,\dots,X_n)$}, if the \emph{conjugate relation}
$$\tau\big(\xi_j P(X_1,\dots,X_n)\big) = (\tau \ootimes \tau)\big((\partial_j P)(X_1,\dots,X_n)\big)$$
holds for each $j=1,\dots,n$ and for all noncommutative polynomials $P\in\C\langle x_1,\dots,x_n\rangle$, where $\tau \ootimes\tau$ denotes the faithful normal tracial state that is induced by $\tau$ on the von Neumann algebra tensor product $\M \ootimes \M$. The conjugate relation implies that such a conjugate system, in case of its existence, is automatically unique; thus, one can define
$$\Phi^\ast(X_1,\dots,X_n) := \sum^n_{j=1} \|\xi_j\|_2^2$$
if a conjugate system $(\xi_1,\dots,\xi_n)$ for $(X_1,\dots,X_n)$ exists, and if there is no conjugate system for $(X_1,\dots,X_n)$, we put $\Phi^\ast(X_1,\dots,X_n) := \infty$.

\subsection{Free differential operators}

Suppose now that $\Phi^\ast(X_1,\dots,X_n) < \infty$ holds and let $(\xi_1,\dots,\xi_n)$ be the conjugate system for $X=(X_1,\dots,X_n)$. It was shown in \cite{MSW17} that $\ev_X: \C\langle x_1,\dots,x_n\rangle \to \C\langle X_1,\dots,X_n\rangle, P \mapsto P(X)$ constitutes under this hypothesis an isomorphism, so that the noncommutative derivatives induce unbounded linear operators
$$\partial_j:\ L^2(\M_0,\tau) \supseteq D(\partial_j) \to L^2(\M_0 \ootimes \M_0, \tau \ootimes \tau)$$
with domain $D(\partial_j) := \C\langle X_1,\dots,X_n\rangle$, which is the unital subalgebra of $\M_0$ generated by $X_1,\dots,X_n$. Since $\partial_j$ is densely defined, we may consider the adjoint operators
$$\partial_j^\ast:\ L^2(\M_0 \ootimes \M_0, \tau \ootimes \tau) \supseteq D(\partial_j^\ast) \to L^2(\M_0,\tau)$$
and we conclude from the conjugate relations that $1\otimes 1 \in D(\partial_j^\ast)$ with $\partial_j^\ast(1\otimes 1) = \xi_j$.

If restricted to its domain, each of the unbounded linear operator $\partial_j$ gives a $\C\langle X_1,\dots,X_n\rangle \otimes \C\langle X_1,\dots,X_n\rangle$-valued derivation on $\C\langle X_1,\dots,X_n\rangle$. 

From $1\otimes 1 \in D(\partial_j^\ast)$, it follows that $\C\langle X_1,\dots,X_n\rangle \otimes \C\langle X_1,\dots,X_n\rangle \subseteq D(\partial_j^\ast)$, which confirms that $\partial_j^\ast$ is densely defined thus yields that $\partial_j$ is closable (cf. \cite[Corollary 4.2]{Voi98}); we denote its closure by $\overline{\partial}_j$.

For each $w\in \C\langle X_1,\dots,X_n\rangle$ and for $j=1,\dots,n$, we have the remarkable bounds
\begin{equation}\label{eq:Dabrowski_bounds}
\|\partial_j(w \otimes 1)\|_2 \leq \|\xi_j\|_2 \|w\| \qquad\text{and}\qquad \|(\id\otimes\tau)(\partial_j w)\|_2 \leq 2 \|\xi_j\|_2 \|w\|,
\end{equation}
which were proven in \cite{Dab10}.

\subsection{Lipschitz conjugate variables}

Suppose again that $\Phi^\ast(X_1,\dots,X_n) < \infty$ and let $(\xi_1,\dots,\xi_n)$ be the conjugate system for $(X_1,\dots,X_n)$. We say that $(\xi_1,\dots,\xi_n)$ are \emph{Lipschitz conjugate variables} for $(X_1,\dots,X_n)$ if the two conditions $\xi_j \in D(\overline{\partial}_j)$ and $\overline{\partial}_j \xi_j \in \M_0 \ootimes \M_0$ are satisfied for each $j=1,\dots,n$.

This notion was introduced in \cite{Dab14}; in \cite{DI16}, it was shown that if $(X_1,\dots,X_n)$ admits Lipschitz conjugate variables, then the von Neumann algebra $\M_0$ that is generated by $X_1,\dots,X_n$ shares many properties with the free group factor $L(\mathbb{F}_n)$. Indeed, freely independent semicircular operators $S_1,\dots,S_n$ generate $L(\mathbb{F}_n)$ and they are the prototypical instance where Lipschitz conjugate variables exist.

If $(\xi_1,\dots,\xi_n)$ are Lipschitz conjugate variables for $X=(X_1,\dots,X_n)$, then necessarily $\xi_1,\dots,\xi_n\in \M_0$; see \cite[Section 5.1]{DI16}. Thus, we may define for $j=1,\dots,n$ the quantities $\gamma_j(X) := \|(\id\otimes\tau)(\overline{\partial}_j \xi_j)\|^{1/2}$ and $\tilde{\gamma}_j(X) := \|\xi_j\| + \|(\id\otimes\tau)(\overline{\partial}_j \xi_j)\|^{1/2}$, as well as
$$\Gamma^\ast(X) := \max_{j=1,\dots,n} (\gamma_j(X) + \tilde{\gamma}_j(X)) = \max_{j=1,\dots,n} \big(\|\xi_j\| + 2 \|(\id\otimes\tau)(\overline{\partial}_j \xi_j)\|^{1/2}\big).$$
As observed in \cite{Dab14}, the assumption of Lipschitz conjugate variables can be used to improve the bounds \eqref{eq:Dabrowski_bounds}; more precisely, according to the version proven in \cite{Mai15}, we have that
\begin{equation}\label{eq:Dabrowski_bounds_Lipschitz}
\|\partial_j^\ast(w \otimes 1)\|_2 \leq \gamma_j(X) \|w\|_2 \quad\text{and}\quad \|(\id\otimes\tau)(\partial_j w)\|_2 \leq \tilde{\gamma}_j(X) \|w\|_2.
\end{equation}

\subsection{Non-microstates free entropy}

It was shown in \cite{Voi98} that arbitrarily small perturbations of any tuple $(X_1,\dots,X_n)$ of selfadjoint operators in $\M$ by freely independent semicircular elements lead to finite non-microstates free Fisher information. Indeed, if $S_1,\dots,S_n$ are semicircular elements in $\M$ which are freely independent among themselves and also free from $\{X_1,\dots,X_n\}$, then \cite[Corollary 6.14]{Voi98} tells us that $(X_1+\sqrt{t}S_n,\dots,X_n+\sqrt{t}S_n)$ admits a conjugate system for each $t>0$ and we have the estimates
\begin{equation}\label{eq:Fisher_perturbation}
\frac{n^2}{C^2 + nt} \leq \Phi^\ast(X_1+\sqrt{t}S_1,\dots,X_n+\sqrt{t}S_n) \leq \frac{n}{t} \quad\text{for all $t>0$},
\end{equation}
where $C\geq 0$ is defined by $C^2 := \tau(X_1^2 + \dots + X_n^2)$; moreover, the function $t\mapsto \Phi^\ast(X_1+\sqrt{t}S_1,\dots,X_n+\sqrt{t}S_n)$, which is defined on $[0,\infty)$ and takes its values in $(0,\infty)$, is decreasing and right continuous.
Based on this observation, Voiculescu introduced in \cite{Voi98} the \emph{non-microstates free entropy} $\chi^\ast(X_1,\dots,X_n)$ of $X_1,\dots,X_n$ by
$$\chi^\ast(X_1,\dots,X_n) := \frac{1}{2} \int^\infty_0\Big(\frac{n}{1+t}-\Phi^\ast(X_1+\sqrt{t}S_1,\dots,X_n+\sqrt{t}S_n)\Big)\, dt + \frac{n}{2}\log(2\pi e).$$
Note that the left inequality in \eqref{eq:Fisher_perturbation} implies in particular that (cf. \cite[Proposition 7.2]{Voi98})
$$\chi^\ast(X_1,\dots,X_n) \leq \frac{n}{2}\log(2\pi e n^{-1} C^2).$$

Of particular interest is the case $n=1$ of a single noncommutative random variable $X=X^\ast \in \M$. It was shown in \cite[Proposition 7.6]{Voi98} that $\chi^\ast(X)$ coincides then with the microstates free entropy $\chi(X)$; for the latter quantity, it was found in \cite[Proposition 4.5]{Voi94} that
\begin{equation}\label{eq:entropy-log_energy}
\chi(X) = -I(\mu_X) + \frac{3}{4} + \frac{1}{2}\log(2\pi)
\end{equation}
holds, where $I(\mu_X)$ denotes the logarithmic energy of the analytic distribution $\mu_X$ of $X$.
Recall that the \emph{logarithmic energy} of a Borel probability measure $\mu$ on $\R$ is defined as
\begin{equation}\label{eq:log_energy}
I(\mu) := \int_\R \int_\R \log\frac{1}{|s-t|} \, d\mu(s)\, d\mu(t).
\end{equation}

\section{H\"older continuity under the assumption of finite free Fisher information}\label{sec:Hoelder_continuity}

Throughout the following, let $(\M,\tau)$ be a tracial $W^\ast$-probability space and let $X_1,\dots,X_n$ be selfadjoint noncommutative random variables in $\M$ that satisfy the regularity condition $\Phi^\ast(X_1,\dots,X_n)<\infty$; whenever we impose the stronger condition of Lipschitz conjugate variables, this will be stated explicitly.

The goal of this section is the proof of Theorem \ref{thm:Hoelder_continuity}. In doing so, we will follow ideas of \cite{CS16}, but with refined arguments similar to \cite{MSY18}. In fact, Theorem \ref{thm:Hoelder_continuity}, in the case $d=1$ of an affine linear polynomial, overlaps with the corresponding result of \cite{MSY18}, if applied to the scalar-valued case; both of them yield the same exponent $\beta=\frac{2}{3}$, which is optimal, as the following example shows.

\begin{example}\label{ex:optimality}
For $\gamma\in(0,1)$, we consider the Borel probability measure $\mu_\gamma$ on $\R$ which is given by $d\mu_\gamma(t) = \rho_\gamma(t)\, dt$ with the density $\rho_\gamma(t) := (1-\gamma) t^{-\gamma} \1_{[0,1]}(t)$. Let $X_\gamma$ be a selfadjoint noncommutative random variable in $(\M,\tau)$ whose distribution is given by $\mu_\gamma$. We know (cf. \cite[Proposition 3.5]{Voi98} and \cite[Proposition 8.18]{MS17}) that $\Phi^\ast(X_\gamma) < \infty$ if and only if $\rho_\gamma \in L^3(\R,dt)$, and the latter condition is satisfied precisely when $\gamma \in (0,\frac{1}{3})$. Moreover, for each $0<\delta\leq 1$, we have that $\mu_\gamma((0,\delta]) = \delta^{1-\gamma}$. Thus, for each $\beta \in (\frac{2}{3},1)$, we find by $X_\gamma$ for any $\gamma \in (1-\beta,\frac{1}{3})$ an operator with $\Phi^\ast(X_\gamma) < \infty$, but whose cumulative distribution function $\F_{X_\gamma}$ cannot be H\"older continuous with exponent $\beta$.
\end{example}

The proof of Theorem \ref{thm:Hoelder_continuity} will be given below, in Subsection \ref{subsec:Hoelder_continuity_proof}. This builds on several previous results, which we collect in Subsection \ref{subsec:Hoelder_continuity_ingredients}.

\subsection{Ingredients for the proof of Theorem \ref{thm:Hoelder_continuity}}\label{subsec:Hoelder_continuity_ingredients}

In this subsection, we lay the groundwork for the proof of Theorem \ref{thm:Hoelder_continuity} in Subsection \ref{subsec:Hoelder_continuity_proof}. We will remind the reader of some facts from free analysis. Most of the material presented here is well-known, but some of these results are slightly modified or extended in order to meet our needs.

\subsubsection{H\"older continuity via spectral projections}

The easy but crucial observation that underlies our approach is the following lemma which is \cite[Lemma 8.3]{MSY18} and which was inspired by \cite{CS16}.

\begin{lemma}\label{lem:Hoelder_criterion}
Let $Y$ be a selfadjoint noncommutative random variable in $(\M,\tau)$. If there exist $c>0$ and $\alpha>1$ such that
$$c \|(Y-s)p\|_2 \geq \|p\|_2^\alpha$$
holds for all $s\in\R$ and each spectral projection $p$ of $Y$, then the cumulative distribution function $\F_Y$ of the analytic distribution $\mu_Y$ of $Y$ is H\"older continuous with exponent $\beta := \frac{2}{\alpha-1}$; more precisely, we have that
$$|\F_Y(t)-\F_Y(s)| \leq c^\beta |t-s|^\beta \qquad\text{for all $s,t\in\R$}.$$
\end{lemma}

For a detailed proof, we refer the interested reader to \cite{MSY18}.

\subsubsection{$L^2$-comparison of left- and right restrictions}

Another ingredient is a nice argument taken from \cite{CS16}; a streamlined version thereof is recorded in the following lemma. Because this is not stated explicitly in \cite{CS16} and since our situation is moreover slightly different, we provide here also the short proof of that statement.

\begin{lemma}\label{lem:comparison}
Let $P\in\C\langle x_1,\dots,x_n\rangle$ be a noncommutative polynomial of degree $d\geq1$. Then, for every non-zero projection $p$ in $\M$, there exists a non-zero projection $q$ in $\M$ such that
$$\tau(q) = \tau(p) \qquad\text{and}\qquad \| P(X_1,\dots,X_n)^\ast q \|_2 = \| P(X_1,\dots,X_n) p\|_2.$$
\end{lemma}

\begin{proof}
Put $Y:=P(X_1,\dots,X_n)$ and consider its polar decomposition $Y = u |Y|$ with a partial isometry $u\in\M$. As $P$ has degree $d\geq1$ and hence is non-constant, we conclude with the results that were obtained in \cite{CS16,MSW17} that $Y$ has no kernel, which finally yields that $u$ is in fact a unitary.
We define $q := upu^\ast$, which is clearly a non-zero projection in $\M$ satisfying $\tau(q) = \tau(p)$. Furthermore, we may check that
$$\|Y^\ast q\|_2 = \| |Y| u^\ast q \|_2 = \| |Y| p u^\ast \|_2 = \| u |Y| p\|_2 = \|Y p\|_2,$$
which concludes the proof.
\end{proof}

We note that the proof given above actually verifies the claim of Lemma \ref{lem:comparison} under the much weaker assumption $\delta^\star(X_1,\dots,X_n)=n$ where $\delta^\star$ is a variant of the non-microstates free entropy dimension defined in \cite[Section 4.1.1]{CS05}; this fact, however, is not needed in the following.

\subsubsection{A quantitative reduction argument}

We recall \cite[Proposition 3.7]{MSW17}. It is this result which allows us to weaken the assumptions that in \cite{CS16} were imposed on $X_1,\dots,X_n$ to finiteness of free Fisher information.

In the sequel, we denote by $(\xi_1,\dots,\xi_n)$ the conjugate system for $X=(X_1,\dots,X_n)$. As before, $\M_0$ will stand for the von Neumann subalgebra of $\M$ that is generated by $X_1,\dots,X_n$, i.e., $\M_0 := \vN(X_1,\dots,X_n)$.

\begin{proposition}\label{prop:Fisher-bound}
Let $P\in\C\langle x_1,\dots,x_n \rangle$ be a (not necessarily selfadjoint) noncommutative polynomial. For all $u,v\in \M_0$, we have that
\begin{equation}\label{eq:Fisher-bound}
|\langle v^\ast (\partial_i P)(X) u, w_1 \otimes w_2 \rangle| \leq 4 \|\xi_i\|_2 \bigl(\|P(X)u\|_2 \|v\| + \|u\| \|P(X)^\ast v\|_2\bigr) \|w_1\| \|w_2\|
\end{equation}
for all $w_1,w_2\in \C\langle X_1,\dots,X_n\rangle$ and $i=1,\dots,n$.
\end{proposition}

The proof of Proposition \ref{prop:Fisher-bound} can be found in \cite{MSW17}; a matrix-valued variant thereof was proven in \cite{MSY18}. In either case, the proof makes heavily use of results from \cite{Voi98} and \cite{Dab10}. An alternative approach building on \cite{CS05} was presented in \cite{CS16}. An extension to the case of more general derivations, with an eye towards free stochastic calculus, is provided in \cite{Mai15}.

We show next an important consequence of Proposition \ref{prop:Fisher-bound}, which will be used in the sequel. For that purpose, let us introduce
\begin{itemize}
 \item for every $v\in\M$ the linear functional $$\phi_v:\ \C\langle x_1,\dots,x_n\rangle \to \C,\quad P \mapsto \tau(v^\ast P(X)),$$
 \item and for every $v\in\M$ and $i=1,\dots,n$ the linear map $$\Delta_{v,i}:\ \C\langle x_1,\dots,x_n\rangle \to \C\langle x_1,\dots,x_n\rangle,\quad P \mapsto (\phi_v \otimes \id)(\partial_i P).$$
\end{itemize}
Note that both $\phi_v$ and $\Delta_{v,i}$ depend implicitly on $X$, but in order to keep the notation as simple as possible, we prefer not to indicate that dependency as $X$ is fixed throughout our discussion.

\begin{corollary}\label{cor:Fisher-bound_reduced}
Let $P\in\C\langle x_1,\dots,x_n \rangle$ be a (not necessarily selfadjoint) noncommutative polynomial. For all $u,v\in\M_0$, we have for $i=1,\dots,n$ that
\begin{equation}\label{eq:Fisher-bound_reduced}
\|(\Delta_{v,i} P)(X) u\|_2^2 \leq 4 \|\xi_i\|_2 \bigl(\|P(X)u\|_2 \|v\| + \|u\| \|P(X)^\ast v\|_2\bigr) \|(\Delta_{v,i} P)(X) u\|
\end{equation}
and
\begin{equation}\label{eq:Fisher-bound_reduced_trace}
|\tau((\Delta_{v,i} P)(X) u)| \leq 4 \|\xi_i\|_2 \bigl(\|P(X)u\|_2 \|v\| + \|u\| \|P(X)^\ast v\|_2\bigr).
\end{equation}
\end{corollary}

\begin{proof}
Take any noncommutative polynomial $w \in \C\langle X_1,\dots,X_n\rangle$. We apply Proposition \ref{prop:Fisher-bound} to $w_1 = 1$ and $w_2 = w$; we derive from \eqref{eq:Fisher-bound} that
$$|\langle (\tau \otimes \id)(v^\ast (\partial_i P)(X)) u, w\rangle| \leq 4 \|\xi_i\|_2 \bigl(\|P(X)u\|_2 \|v\| + \|u\| \|P(X)^\ast v\|_2\bigr) \|w\|.$$
Now, by Kaplansky's density theorem, as $\C\langle X_1,\dots,X_n\rangle$ is strongly dense in $\M_0$, the latter inequality extends to all $w\in\M_0$; thus, we may we apply it to
$$w := (\tau\otimes\id)\big(v^\ast (\partial_i P)(X)\big) u = (\Delta_{v,i} P)(X) u$$
for any fixed $i=1,\dots,n$; in this way, we obtain the inequality \eqref{eq:Fisher-bound_reduced}.

The second inequality \eqref{eq:Fisher-bound_reduced_trace} follows directly from the inequality \eqref{eq:Fisher-bound} given in Proposition \ref{prop:Fisher-bound} if the latter is applied to $w_1 = 1$ and $w_2 = 1$.
\end{proof}

In the case of Lipschitz conjugate variables, we can strengthen Corollary \ref{cor:Fisher-bound_reduced} as follows.

\begin{proposition}\label{prop:Lipschitz-bound_reduced}
Suppose that $X=(X_1,\dots,X_n)$ admits Lipschitz conjugate variables $(\xi_1,\dots,\xi_n)$. Then, for all (not necessarily selfadjoint) noncommutative polynomials $P\in\C\langle x_1,\dots,x_n \rangle$ and for all $u,v\in\M_0$, we have for $i=1,\dots,n$ that
\begin{equation}\label{eq:Lipschitz-bound_reduced}
\|(\Delta_{v,i} P)(X) u\|_2^2 \leq \bigl(\gamma_i(X) \|P(X)u\|_2 \|v\|_2 + \tilde{\gamma}_i(X) \|u\|_2 \|P(X)^\ast v\|_2\bigr) \|(\Delta_{v,i} P)(X) u\|,
\end{equation}
and
\begin{equation}\label{eq:Lipschitz-bound_reduced_trace}
|\tau((\Delta_{v,i} P)(X) u)| \leq \gamma_i(X) \|P(X)u\|_2 \|v\|_2 + \tilde{\gamma}_i(X) \|u\|_2 \|P(X)^\ast v\|_2.
\end{equation}
\end{proposition}

\begin{proof}
To begin with, we fix $i=1,\dots,n$ and we suppose that $u,v\in\C\langle X_1,\dots,X_n\rangle$. We recall from the proof of \cite[Proposition 3.7]{MSW17}, which is Proposition \ref{prop:Fisher-bound} given above, that in this case
$$\langle (\Delta_{v,i} P) u, w \rangle = \langle P(X) u, \partial_i^\ast(vw \otimes 1) \rangle - \langle (\id \otimes \tau)(\partial_i u), P(X)^\ast v w\rangle$$
for all $w\in\C\langle X_1,\dots,X_n\rangle$ and $i=1,\dots,n$. Using the bounds \eqref{eq:Dabrowski_bounds_Lipschitz}, we arrive at
\begin{align*}
|\langle (\Delta_{v,i} P) u, w \rangle|
&\leq \|P(X) u\|_2 \|\partial_i^\ast(vw \otimes 1)\|_2 + \|(\id \otimes \tau)(\partial_i u)\|_2 \|P(X)^\ast v\|_2 \|w\|\\
&\leq \big(\gamma_i(X) \|P(X) u\|_2 \|v\|_2 + \tilde{\gamma}_i(X) \|u\|_2 \|P(X)^\ast v\|_2\big) \|w\|.
\end{align*}
Now, we proceed as in the proof of Corollary \ref{cor:Fisher-bound_reduced}: we use the Kaplansky density theorem to verify that the latter inequality extends to hold for all $w\in\M_0$, so that we may apply it to $w = (\Delta_{v,i} P)(X) u$; this gives \eqref{eq:Lipschitz-bound_reduced}, whereas \eqref{eq:Lipschitz-bound_reduced_trace} is obtained for $w = 1$.
\end{proof}

\subsubsection{A Bernstein type inequality for noncommutative derivatives}

The proof of Theorem \ref{thm:Hoelder_continuity}, which will be given in Section \ref{subsec:Hoelder_continuity_proof}, relies on an iteration of the estimates provided in Corollary \ref{cor:Fisher-bound_reduced} and Proposition \ref{prop:Lipschitz-bound_reduced}. In doing so, it will be important to control the operator norm of expressions like $(\Delta_{q_k, i_k} \cdots \Delta_{q_1,i_1} P)(X)$ for projections $q_1,\dots,q_k\in\M_0$ and any noncommutative polynomial $P$.
We will achieve this in four steps; the crucial ingredient will be Proposition \ref{prop:Bernstein}, which can be seen an analogue of Bernstein's inequality for noncommutative polynomials.

\paragraph{Step 0}

On $\C\langle x_1,\dots,x_n\rangle$, we may define, for any fixed $R>0$, a norm $\|\cdot\|_R$ by putting
$$\|P\|_R := \sum^d_{k=0} \sum_{1\leq i_1,\dots,i_k\leq n} |a_{i_1,\dots,i_k}| R^k$$
for each $P\in\C\langle x_1,\dots,x_n\rangle$ that is written in the form \eqref{eq:ncpoly}. It is easily seen that
\begin{equation}\label{eq:evaluation_bound}
\|P(X)\| \leq \|P\|_R \qquad\text{for all $P\in\C\langle x_1,\dots,x_n\rangle$}
\end{equation}
holds, whenever the condition
\begin{equation}\label{eq:radius_constraint}
R \geq \max_{i=1,\dots,n} \|X_i\|
\end{equation}
is satisfied. Therefore, in order to control $\|P(X)\|$, it suffices to provide bounds for $\|P\|_R$ for any $R>0$ satisfying \eqref{eq:radius_constraint}.

\paragraph{Step 1}

Using the norm $\|\cdot\|_R$ for any given $R>0$, we may introduce on $\C\langle x_1,\dots,x_n\rangle \otimes \C\langle x_1,\dots,x_n\rangle$, the space of all noncommutative bi-polynomials, the associated projective norm $\|\cdot\|_{R,\pi}$ by
$$\|Q\|_{R,\pi} = \inf\bigg\{\sum^m_{k=1} \|Q_{1,k}\|_R \|Q_{2,k}\|_R \mathrel{\bigg|} Q = \sum^m_{k=1} Q_{1,k} \otimes Q_{2,k}\bigg\}.$$

\begin{lemma}\label{lem:projective_norm_bound_1}
Let $q\in\M_0$ be any projection and suppose that $R>0$ is chosen such that \eqref{eq:radius_constraint} holds. Then the following holds true:
\begin{enumerate}
 \item\label{it:projective_norm_bound_1-i} The functional $\phi_q$ is positive and satisfies for all $P\in\C\langle x_1,\dots,x_n\rangle$
 $$|\phi_q(P)| \leq \tau(q) \|P\|_R.$$
 \item\label{it:projective_norm_bound_1-ii} For each $Q\in\C\langle x_1,\dots,x_n\rangle \otimes \C\langle x_1,\dots,x_n\rangle$, we have that
 $$\|(\phi_q \otimes \id)(Q)\|_R \leq \tau(q) \|Q\|_{R,\pi}.$$
 \item\label{it:projective_norm_bound_1-iii} For every $P\in\C\langle x_1,\dots,x_n\rangle$ and $i=1,\dots,n$, we have that
 $$\|\Delta_{q,i} P\|_R \leq \tau(q) \|\partial_i P\|_{R,\pi}.$$
\end{enumerate}
\end{lemma}

\begin{proof}
(i) Let $P\in\C\langle x_1,\dots,x_n\rangle$ be any noncommutative polynomial. Suppose that $P$ is written in the form \eqref{eq:ncpoly}. Then
$$\phi_q(P) = \sum^d_{k=0} \sum_{1\leq i_1,\dots,i_k\leq n} a_{i_1,\dots,i_k} \tau(q X_{i_1} \cdots X_{i_k}).$$
By H\"older's inequality we see that $|\tau(q X_{i_1} \cdots X_{i_k})| \leq \|q\|_1 \|X_{i_1} \cdots X_{i_k}\| \leq \|q\|_1 R^k$, where $\|q\|_1 = \tau(q)$ as $q$ is a projection. Thus, in summary, we obtain as claimed that
$$|\phi_q(P)| \leq \tau(q) \bigg(\sum^d_{k=0} \sum_{1\leq i_1,\dots,i_k\leq n} |a_{i_1,\dots,i_k}| R^k\bigg) = \tau(q) \|P\|_R.$$

(ii) Take any noncommutative bi-polynomial $Q$. Then, by \ref{it:projective_norm_bound_1-i},
$$\|(\phi_q \otimes \id)(Q)\|_R \leq \sum^m_{k=1} |\phi_q(Q_{1,k})| \|Q_{2,k}\|_R \leq \tau(q) \sum^m_{k=1} \|Q_{1,k}\|_R \|Q_{2,k}\|_R,$$
and by passing to the infimum over all representations $Q = \sum^m_{k=1} Q_{1,k} \otimes Q_{2,k}$ of $Q$, we finally arrive at the assertion.

(iii) Since $\Delta_{q,i} P = (\phi_q \otimes \id)(\partial_i P)$, applying \ref{it:projective_norm_bound_1-ii} to $Q=\partial_i P$ yields directly the claim.
\end{proof}

The essence of Lemma \ref{lem:projective_norm_bound_1} is Item \ref{it:projective_norm_bound_1-iii}, which will allow us to derive bounds for $\|\Delta_{q,i} P\|_R$ from bounds for $\|\partial_i P\|_{R,\pi}$.

\paragraph{Step 2}

For the purpose of estimating $\|\partial_i P\|_{R,\pi}$ against $\|P\|_R$, we have to restrict attention to subspaces of $\C\langle x_1,\dots,x_n\rangle$ consisting of all noncommutative polynomials with degree below a given threshold; more precisely, for every $d\geq 0$, we work with the subspace of $\C\langle x_1,\dots,x_n\rangle$ that is given by
$$\P_d := \Big\{P\in \C\langle x_1,\dots,x_n\rangle \mathrel{\Big|} \deg(P) \leq d\Big\}.$$
On $\P_d$, we have the following estimate, which is a variant of a result that can be found in \cite[Section 4]{Voi98}.

\begin{lemma}\label{lem:projective_norm_bound_2}
Take any $R>0$ that satisfies \eqref{eq:radius_constraint}. Then, for each $P\in\P_d$ and $i=1,\dots,n$, it holds true that
$$\|\partial_i P\|_{R,\pi} \leq \frac{d}{R} \|P\|_R.$$
\end{lemma}

\begin{proof}
Take any $P\in\P_d$ that is written in the form \eqref{eq:ncpoly}. Then, for $i=1,\dots,n$, we have by definition of the noncommutative derivatives that
$$\partial_i P = \sum^d_{k=1} \sum_{1\leq i_1,\dots,i_k\leq n} \sum^k_{j=1} \delta_{i,i_j} a_{i_1,\dots,i_k} x_{i_1} \cdots x_{i_{j-1}} \otimes x_{i_{j+1}} \cdots x_{i_k}.$$
Therefore, we may conclude that
\begin{align*}
\|\partial_i P\|_{R,\pi}
&\leq \sum^d_{k=1} \sum_{1\leq i_1,\dots,i_k\leq n} \sum^k_{j=1} \delta_{i,i_j} |a_{i_1,\dots,i_k}| \|x_{i_1} \cdots x_{i_{j-1}}\|_R \|x_{i_{j+1}} \cdots x_{i_k}\|_R\\
&\leq \sum^d_{k=1} \sum_{1\leq i_1,\dots,i_k\leq n} k |a_{i_1,\dots,i_k}| R^{k-1}= \frac{d}{R} \|P\|_R,
\end{align*}
which is the asserted inequality.
\end{proof}

\paragraph{Step 3}

By combining Lemma \ref{lem:projective_norm_bound_2} with Item \ref{it:projective_norm_bound_1-iii} of Lemma \ref{lem:projective_norm_bound_1}, we see that $\|\Delta_{q,i} P\|_R \leq \tau(q) \|\partial_i P\|_{R,\pi}$ and $\|\partial_i P\|_{R,\pi} \leq \frac{d}{R} \|P\|_R$ hold under the assumption \eqref{eq:radius_constraint} for every $P\in \P_d$, each projection $q\in\M_0$ and for $i=1,\dots,n$. Putting this together yields that
\begin{equation}\label{eq:norm_bound}
\|\Delta_{q,i} P\|_R \leq d \frac{\tau(q)}{R} \|P\|_R.
\end{equation}
This enables us to control expressions like $\Delta_{q_k, i_k} \Delta_{q_{k-1},i_{k-1}} \cdots \Delta_{q_1,i_1} P$ by iterating the latter estimate \eqref{eq:norm_bound}; this is the content of the next proposition.

\begin{proposition}\label{prop:Bernstein}
For a $k\in\N$, let $q_1,\dots,q_{k-1}\in\M_0$ be arbitrary projections and let $1\leq i_1,\dots,i_k \leq n$ be any collection of indices. Moreover, let $R>0$ be such that \eqref{eq:radius_constraint} is satisfied. Then
$$\|\Delta_{q_k, i_k} \cdots \Delta_{q_1,i_1} P\|_R \leq \frac{d!}{(d-k)!} \frac{\tau(q_1) \cdots   \tau(q_k)}{R^k} \|P\|_R$$
holds for every noncommutative polynomial $P\in\P_d$.
\end{proposition}

\begin{proof}
We proceed by mathematical induction on $k$. In the case $k=1$, the asserted estimate is nothing but \eqref{eq:norm_bound}. The induction then follows by noting that $\Delta_{q_1,i_1} P \in \P_{d-1}$ and using the bound in \eqref{eq:norm_bound}.
\end{proof}

By combining Proposition \ref{prop:Bernstein} with \eqref{eq:evaluation_bound}, we obtain immediately the following corollary, which is the desired bound that we will use in the proof of Theorem \ref{thm:Hoelder_continuity}.

\begin{corollary}\label{cor:norm_bound}
In the situation of Proposition \ref{prop:Bernstein}, we have furthermore that
$$\|(\Delta_{q_k, i_k} \Delta_{q_{k-1},i_{k-1}} \cdots \Delta_{q_1,i_1} P)(X)\| \leq \frac{d!}{(d-k)!} \frac{\tau(q_1) \cdots \tau(q_{k-1}) \tau(q_k)}{R^k} \|P\|_R.$$
\end{corollary}

\subsection{Proof of Theorem \ref{thm:Hoelder_continuity}}\label{subsec:Hoelder_continuity_proof}

Let us fix any selfadjoint noncommutative polynomial $P\in \C\langle x_1,\dots,x_n\rangle$ that has degree $d := \deg(P) \geq 1$. Accordingly, $P$ belongs to the space $\P_d$; we suppose that $P$ is written in the form \eqref{eq:ncpoly}. Let us fix some leading coefficient $a_{i_1,\dots,i_d}$ of $P$ that is non-zero. Further, we choose $R>0$ such that $R \geq \max_{i=1,\dots,n} \|X_i\|$, i.e., \eqref{eq:radius_constraint} holds.

We will prove the result for $\Phi^\ast(X) < \infty$ together with its strengthening in the case of Lipschitz conjugate variables. In either instance, the H\"older continuity of $\mu_Y$ for the noncommutative random variable $Y=P(X)$ will follow from Lemma \ref{lem:Hoelder_criterion}; for that purpose, we are going to prove that there are $\alpha>1$ and $c>0$ such that $Y$ satisfies
\begin{equation}\label{eq:Hoelder_continuity_proof}
c \|(Y-s) p\|_2 \geq \|p\|_2^\alpha
\end{equation}
for every $s\in\R$ and every projection $p\in\M_0$; note that it clearly suffices to consider the case $p\neq 0$.

Correspondingly, let us take now any $s\in\R$ and any non-zero projection $p\in\M_0$. Put $P_0 := P-s$; note that $\deg(P_0) = \deg(P) \geq 1$. We construct then recursively, for every $1\leq k\leq d$, a non-zero projection $q_k\in\M_0$ and a noncommutative polynomial $P_k\in\C\langle x_1,\dots,x_n\rangle$ which is non-constant for $1 \leq k < d$ according to the following rules:
\begin{enumerate}
 \item With the help of Lemma \ref{lem:comparison}, applied to the non-constant polynomial $P_{k-1}$, we construct the projection $q_k\in\M_0$ so that $$\tau(q_k) = \tau(p) \qquad\text{and}\qquad \|P_{k-1}(X)^\ast q_k\|_2 = \|P_{k-1}(X) p\|_2.$$

 \item Subsequently, we put $P_k := \Delta_{q_k,i_k} P_{k-1} = \Delta_{q_k,i_k} \dots \Delta_{q_1,i_1} P$, which is for $1\leq k < d$ again a non-constant polynomial.
 
 In fact, we have $\deg(P_k) = d-k$, since necessarily $\deg(P_k) \leq d-k$, due to the iterative application of noncommutative derivatives, and since the monomial $x_{i_{k+1}} \dots x_{i_d}$ of degree $d-k$ shows up in $P_k$ with the non-zero coefficient $\tau(p)^k a_{i_1,\dots,i_d}$.
\end{enumerate}

Involving now the inequality \eqref{eq:Fisher-bound_reduced} provided in Corollary \ref{cor:Fisher-bound_reduced}, we infer that for $k=1,\dots,d$
$$\|P_k(X) p\|_2^2 \leq 8 \|\xi_{i_k}\|_2 \|P_k(X) p\| \|P_{k-1}(X) p\|_2,$$
whereas in the case of Lipschitz conjugate variables
$$\|P_k(X) p\|_2^2 \leq (\gamma_{i_k}(X) + \tilde{\gamma}_{i_k}(X)) \|P_k(X) p\| \|p\|_2 \|P_{k-1}(X) p\|_2$$
holds, as one sees by using instead the inequality \eqref{eq:Lipschitz-bound_reduced} provided in Proposition \ref{prop:Lipschitz-bound_reduced}.

We estimate $\|\xi_{i_k}\|_2 \leq \Phi^\ast(X)^{1/2}$ and $\gamma_{i_k}(X) + \tilde{\gamma}_{i_k}(X) \leq \Gamma^\ast(X)$; moreover, by using Corollary \ref{cor:norm_bound}, respectively, we get that
$$\|P_k(X) p\| \leq \|(\Delta_{q_k,i_k} \dots \Delta_{q_1,i_1} P)(X)\| \leq \frac{d!}{(d-k)!} \frac{\tau(q_1) \cdots \tau(q_k)}{R^k} \|P\|_R.$$
Because $\tau(p) = \|p\|_2^2$, this yields in summary for every $k=1,\dots,d-1$
$$\|P_k(X) p\|_2^2 \leq c_k  \|P_{k-1}(X) p\|_2 \quad\text{with}\quad c_k := 8 \Phi^\ast(X)^{1/2} \frac{d!}{(d-k)!} \frac{\|p\|_2^{2k}}{R^k} \|P\|_R$$
and
$$\|P_k(X) p\|_2^2 \leq \tilde{c}_k  \|P_{k-1}(X) p\|_2 \quad\text{with}\quad \tilde{c}_k := \Gamma^\ast(X) \frac{d!}{(d-k)!} \frac{\|p\|_2^{2k+1}}{R^k} \|P\|_R,$$
respectively.
Further, by the second inequality \eqref{eq:Fisher-bound_reduced_trace} of Corollary \ref{cor:Fisher-bound_reduced}
$$|\tau(P_d(X) p)| \leq  c_d \|P_{d-1}(X) p\|_2 \qquad\text{with}\qquad c_d := 8 \Phi^\ast(X)^{1/2},$$
and respectively, by the second inequality \eqref{eq:Lipschitz-bound_reduced_trace} of Proposition \ref{prop:Lipschitz-bound_reduced},
$$|\tau(P_d(X) p)| \leq  \tilde{c}_d \|P_{d-1}(X) p\|_2 \qquad\text{with}\qquad \tilde{c}_d := \Gamma^\ast(X) \|p\|_2.$$

By iterating the latter inequalities, we obtain that
\begin{align*}
|\tau(P_d(X) p)|^{2^{d-1}} &\leq \bigg(\prod^{d}_{k=1} c_k^{2^{k-1}}\bigg) \|P_0(X) p\|_2 \qquad\text{and}\\
|\tau(P_d(X) p)|^{2^{d-1}} &\leq \bigg(\prod^{d}_{k=1} \tilde{c}_k^{2^{k-1}}\bigg) \|P_0(X) p\|_2,
\end{align*}
respectively. With the help of the formulas
$$\sum^{d-1}_{k=1} k2^{k-1} = (d-2)2^{d-1}+1 \quad\text{and}\quad \sum^{d-1}_{k=1} (2k+1)2^{k-1} = (2d-3)2^{d-1}+1,$$
the involved products can be simplified to
$$\prod^{d}_{k=1} c_k^{2^{k-1}} = \big(8 \Phi^\ast(X)^{1/2}\big)^{2^d-1} \|P\|_R^{2^{d-1}-1} \frac{\|p\|_2^{(d-2)2^d+2}}{R^{(d-2)2^{d-1}+1}}\prod^{d-1}_{k=1} \Big(\frac{d!}{(d-k)!}\Big)^{2^{k-1}}$$
and
$$\prod^{d}_{k=1} \tilde{c}_k^{2^{k-1}} =  \Gamma^\ast(X)^{2^d-1} \|P\|_R^{2^{d-1}-1} \frac{\|p\|_2^{(d-1)2^d+1}}{R^{(d-2)2^{d-1}+1}}\prod^{d-1}_{k=1} \Big(\frac{d!}{(d-k)!}\Big)^{2^{k-1}},$$
respectively. Note that $P_0 = P-s$ and $P_d = \Delta_{q_d,i_d} \dots \Delta_{q_1,i_1} P = \tau(p)^d a_{i_1,\dots,i_d}$ as $P$ has degree $d$; thus
$$\|P_0(X) p\|_2 = \|(Y-s)p\|_2 \quad\text{and}\quad |\tau(P_d(X) p)|^{2^{d-1}} = |a_{i_1,\dots,i_d}|^{2^{d-1}} \|p\|_2^{(d+1)2^d}.$$ 
We conclude now that \eqref{eq:Hoelder_continuity_proof} holds with $\alpha = (d+1)2^d - (d-2)2^d - 2 = 3\cdot 2^d - 2$ and
\begin{equation}\label{eq:Hoelder_constant-1a}
c = \big(8 \Phi^\ast(X)^{1/2}\big)^{2^d-1} \frac{\|P\|_R^{2^{d-1}-1}}{|a_{i_1,\dots,i_d}|^{2^{d-1}} R^{(d-2)2^{d-1}+1}} \prod^{d-1}_{k=1} \Big(\frac{d!}{(d-k)!}\Big)^{2^{k-1}},
\end{equation}
whereas \eqref{eq:Hoelder_continuity_proof} holds with $\tilde{\alpha} = (d+1)2^d - (d-1)2^d - 1 = 2^{d+1} - 1$ and
\begin{equation}\label{eq:Hoelder_constant-1b}
\tilde{c} = \Gamma^\ast(X) ^{2^d-1} \frac{\|P\|_R^{2^{d-1}-1}}{|a_{i_1,\dots,i_d}|^{2^{d-1}} R^{(d-2)2^{d-1}+1}} \prod^{d-1}_{k=1} \Big(\frac{d!}{(d-k)!}\Big)^{2^{k-1}}
\end{equation}
in the case of Lipschitz conjugate variables.

Now, using Lemma \ref{lem:Hoelder_criterion}, we see that $\F_Y$ is H\"older continuous with exponent $\beta = \frac{2}{\alpha-1} = \frac{2}{3(2^d-1)}$ and $\tilde{\beta} = \frac{2}{\tilde{\alpha}-1} = \frac{1}{2^d-1}$, respectively; the associated H\"older constants are given by $C = c^\beta$ and $\tilde{C} = \tilde{c}^{\tilde{\beta}}$. This concludes the proof of Theorem \ref{thm:Hoelder_continuity}.

\subsection{More about the H\"older constant}

We take now a closer look at the constants $c$ and $\tilde{c}$ given in \eqref{eq:Hoelder_constant-1a} and \eqref{eq:Hoelder_constant-1b}, respectively. Besides $\|P\|_R$, we can extract from there another quantity that depends solely on $R$ and the algebraic structure of $P$. 
More precisely, for any given $R>0$, we define for every noncommutative polynomial $0\neq P\in \C\langle x_1,\dots,x_n\rangle$ its \emph{leading weight} $\rho_R(P) \in (0,1]$ by
$$\rho_R(P) := \max_{1\leq i_1,\dots,i_d \leq n} \frac{|a_{i_1,\dots,i_d}| R^d}{\|P\|_R}, \qquad\text{where $d:=\deg(P)$}.$$

Using this quantity, we can rearrange the terms appearing in \eqref{eq:Hoelder_constant-1a} as
$$c = \rho_R(P)^{-2^{d-1}} \big(8 R \Phi^\ast(X)^{1/2}\big)^{2^d-1} \frac{1}{\|P\|_R} \prod^{d-1}_{k=1} \Big(\frac{d!}{(d-k)!}\Big)^{2^{k-1}}.$$
Since the explicit value for the H\"older constant $C>0$ of $\F_Y$ that we found in the proof of Theorem \ref{thm:Hoelder_continuity} is $C = c^\beta$ with $\beta = \frac{2}{3(2^d-1)}$, we infer from the latter that 
\begin{equation}\label{eq:Hoelder_constant-2a}
C = C_d^\frac{2}{3} \rho_R(P)^{-\frac{2^d}{3(2^d-1)}} \big(8 R \Phi^\ast(X)^{1/2}\big)^{\frac{2}{3}} \|P\|_R^{-\frac{2}{3(2^d-1)}},
\end{equation}
where $C_d$ is the numerical quantity depending only on $d$ which is given by
\begin{equation}\label{eq:constant}
C_d := \bigg(\prod^{d-1}_{k=1} \Big(\frac{d!}{(d-k)!}\Big)^{2^{k-1}}\bigg)^{\frac{1}{2^d-1}}.
\end{equation}

Likewise, we may rearrange the terms in \eqref{eq:Hoelder_constant-1b} as
$$\tilde{c} = \rho_R(P)^{-2^{d-1}} \big(R \Gamma^\ast(X)\big)^{2^d-1} \frac{1}{\|P\|_R} \prod^{d-1}_{k=1} \Big(\frac{d!}{(d-k)!}\Big)^{2^{k-1}},$$
and since the proof of Theorem \ref{thm:Hoelder_continuity} gives $\tilde{C} = \tilde{c}^{\tilde{\beta}}$ with $\tilde{\beta} = \frac{1}{2^d-1}$, we obtain that
\begin{equation}\label{eq:Hoelder_constant-2b}
\tilde{C} = C_d \rho_R(P)^{-\frac{2^{d-1}}{2^d-1}} \big(R \Gamma^\ast(X)\big) \|P\|_R^{-\frac{1}{2^d-1}},
\end{equation}
where $C_d$ is the numerical quantity defined in \eqref{eq:constant}.

It is natural to ask for the order by which $C_d$ grows with $d$; this question is addressed in the next lemma.

\begin{lemma}\label{lem:constant_bound}
For every $d\in\N$, the constant $C_d$ from \eqref{eq:constant} satisfies $(d!)^{1/2} \leq C_d \leq d!$. 
\end{lemma}

\begin{proof}
Since $C_1=1$, the assertion is trivially true in the case $d=1$. Thus, assume from now on that $d\geq 2$. It is straightforward then to check that
$$\log(C_d) = \frac{1}{2^d-1} \sum^{d-1}_{k=1} \sum_{l=d-k}^{d-1} 2^k \log(l+1) = \frac{2^d}{2^d-1} \Bigg(\log(d!) - \sum^{d-1}_{l=1} 2^{-l} \log(l+1)\Bigg).$$
From the latter, we easily deduce that
\begin{align*}
\log(C_d) &\leq \frac{2(2^{d-1}-1)}{2^d-1} \log(d!) \leq \log(d!) \qquad\text{and}\\
\log(C_d) &\geq \frac{2^{d-1}}{2^d-1} \log(d!) \geq \frac{1}{2} \log(d!),
\end{align*}
which proves the assertion.
\end{proof}

\section{H\"older continuity and finite free entropy}\label{sec:finite_entropy}

This section is devoted to the proof of Theorem \ref{thm:finite_entropy}. In fact, we will prove the following theorem, which provides an explicit upper bound for the logarithmic energy (as defined in \eqref{eq:log_energy}) of the analytic distribution of the considered polynomial evaluation. Thanks to \eqref{eq:entropy-log_energy}, the latter results directly in a lower bound for both the microstates and the non-microstates free entropy; this, in particular, verifies the assertion of Theorem \ref{thm:finite_entropy}.

\begin{theorem}\label{thm:finite_entropy_bound}
Let $(\M,\tau)$ be a tracial $W^\ast$-probability space and suppose that $X_1,\dots,X_n$ are selfadjoint noncommutative random variables in $\M$ satisfying $\Phi^\ast(X_1,\dots,X_n)<\infty$. Further, let $P\in\C\langle x_1,\dots,x_n\rangle$ be any selfadjoint noncommutative polynomial of degree $d\geq 1$ and consider the associated selfadjoint noncommutative random variable $Y:=P(X_1,\dots,X_n)$ in $\M$. Then the analytic distribution $\mu_Y$ of $Y$ has finite logarithmic energy $I(\mu_Y)$ that can be bounded from above by
\begin{equation}\label{eq:log_energy_bound}
I(\mu_Y) \leq 3(2^d-1) C_d^\frac{2}{3} \rho_R(P)^{-\frac{2^d}{3(2^d-1)}} \big(8 R \Phi^\ast(X)^{1/2}\big)^{\frac{2}{3}} \|P\|_R^{-\frac{2}{3(2^d-1)}},
\end{equation}
where $C_d$ is the constant introduced in \eqref{eq:constant}. 
\end{theorem}

Using \cite{Jam15}, Theorem \ref{thm:finite_entropy_bound} follows rather immediately from Theorem \ref{thm:Hoelder_continuity}. To be more precise, it was shown in \cite{Jam15} that for every Borel probability measure $\mu$ on $\R$ that has a cumulative distribution function $\F_\mu$ which is H\"older continuous with exponent $\beta\in(0,1]$ and with a H\"older constant $C>0$, i.e., if $\F_\mu$ satisfies \eqref{eq:Hoelder_condition}, then the logarithmic energy of $\mu$ can be bounded from above by
\begin{equation}\label{eq:Jam_bound}
I(\mu) \leq 2\frac{C}{\beta}.
\end{equation}

\begin{proof}[Proof of Theorem \ref{thm:finite_entropy_bound}]
Using Theorem \ref{thm:Hoelder_continuity}, we see that $\F_Y$ satisfies \eqref{eq:Hoelder_condition} with the constant $C$ given by \eqref{eq:Hoelder_constant-2a} and $\beta = \frac{2}{3(2^d - 1)}$. Thus, the asserted bound \eqref{eq:log_energy_bound} follows from \eqref{eq:Jam_bound}.
\end{proof}

\section{Convergence in distribution and the Kolmogorov distance}\label{sec:Kolmogorov}

Among the strongest metrics that are usually studied on the space of all Borel probability measures on the real line $\R$ is the so-called \emph{Kolmogorov distance}; this metric $\Delta$ is defined for any two Borel probability measures $\mu$ and $\nu$ on $\R$ by
$$\Delta(\mu,\nu) := \sup_{t\in\R} |\F_\mu(t) - \F_\nu(t)|.$$
Though its definition is quite appealing, convergence with respect to the Kolmogorov distance is much more rigid than, for instance, convergence with respect to the so-called \emph{L\'evy distance}. The latter is defined by
$$L(\mu,\nu) := \inf\{\epsilon>0 \mid \forall t\in\R:\ \F_\mu(t-\epsilon) - \epsilon \leq \F_\nu(t) \leq \F_\mu(t+\epsilon) + \epsilon\}$$
and is known to provide a metrization of convergence in distribution.

It is accordingly a challenging task to control the Kolmogorov distance in concrete situations. In view of our regularity results, some known ``self-improvement'' phenomenon is worth mentioning: if convergence towards a measure with H\"older continuous cumulative distribution function is considered, then convergence in distribution automatically implies convergence in Kolmogorov distance; see Theorem \ref{thm:convergence_in_distribution} below.

The drawback of this approach, however, is that it does not give rates of convergence for the Kolmogorov distance if the convergence is measured only in terms of the associated Cauchy-Stieltjes transforms. Based on estimates derived in \cite{Bai93a,Bai93b} (see also \cite{BS10}), we provide here with Theorem \ref{thm:Hoelder_criterion} a criterion that gives explicitly such rates in general situations.

\subsection{Convergence in Kolmogorov distance}

Let us denote by $\C^\pm$ the complex upper respectively lower half-plane, i.e., $\C^\pm := \{z\in\C \mid \pm\Im(z)>0\}$. To each Borel probability measure $\mu$ on the real line $\R$, we may associate its \emph{Cauchy transform}, i.e., the holomorphic function $G_\mu: \C^+ \to \C^-$ that is given by
$$G_\mu(z) := \int_\R \frac{1}{z-t}\, d\mu(t) \qquad\text{for all $z\in\C^+$}.$$

Let us first recall the following well-known facts that are well surveyed in \cite{GH03}.

\begin{theorem}\label{thm:convergence_in_distribution}
Let $(\mu_n)_{n=1}^\infty$ be a sequence of Borel probability measures on $\R$ and let $\nu$ be another Borel probability measure on $\R$. Then the following statements are equivalent:
\begin{enumerate}
 \item\label{it:convergence_in_distribution-i} $(\mu_n)_{n=1}^\infty$ converges in distribution to $\nu$.
 \item\label{it:convergence_in_distribution-ii} We have that $(G_{\mu_n})_{n=1}^\infty$ converges uniformly on compact subsets of $\C^+$ to $G_\nu$.
 \item\label{it:convergence_in_distribution-iii} There is an infinite subset $K \subseteq \C^+$ with an accumulation point in the complex upper half-plane $\C^+$ such that $G_{\mu_n}(z) \to G_\nu(z)$ as $n\to\infty$ for each $z\in K$. 
\end{enumerate}
If we assume in addition that the target measure $\nu$ has a cumulative distribution function $\F_\nu$ that is H\"older continuous with some exponent $\beta\in(0,1]$, then the above statements \ref{it:convergence_in_distribution-i}, \ref{it:convergence_in_distribution-ii}, and \ref{it:convergence_in_distribution-iii} are equivalent also to
\begin{enumerate}\addtocounter{enumi}{3}
 \item\label{it:convergence_in_distribution-iv} We have $\Delta(\mu_n,\nu) \to 0$ as $n\to\infty$.
\end{enumerate}
\end{theorem}

If we require the target measure $\nu$ to have a cumulative distribution function $\F_\nu$ that is H\"older continuous with exponent $\beta\in(0,1]$ and a H\"older constant $C>0$, then \cite[Lemma 12.18]{BS10} says that
$$L(\mu_n,\nu) \leq \Delta(\mu_n,\nu) \leq (C+1) L(\mu_n,\nu)^\beta,$$
from which the equivalence of \ref{it:convergence_in_distribution-i} and \ref{it:convergence_in_distribution-iv} follows, since one has $L(\mu_n,\nu) \to 0$ as $n\to\infty$ if and only if \ref{it:convergence_in_distribution-i} holds.

Here, we will prove the following quantitative version of Theorem \ref{thm:convergence_in_distribution}. We will denote by $\bS_\rho$ for any $0<\rho\leq \infty$ the strip $\{z\in\C \mid 0 < \Im(z) < \rho\}$ in $\C^+$; clearly, $\bS_\infty = \C^+$.

\begin{theorem}\label{thm:Hoelder_criterion}
Let $(\mu_n)_{n=1}^\infty$ be a sequence of Borel probability measures on $\R$ and let $\nu$ be any other Borel probability measure on $\R$. Suppose the following:
\begin{enumerate}
 \item\label{it:cond-1} The cumulative distribution function $\F_\nu$ of the measure $\nu$ is H\"older continuous with exponent $\beta\in(0,1]$ and a H\"older constant $C>0$.
 \item\label{it:cond-2} There are continuous functions $\Theta: \bS_\rho \to [0,\infty)$ for some $0<\rho \leq \infty$ and $\Theta_0: [0,\infty) \to [0,\infty)$ that satisfy the growth conditions $$\limsup_{R\to\infty} R^{-l} \max_{r\in[0,R]} \Theta_0(r) < \infty$$ for some $l\geq 0$ and $$\Theta(z) \leq \frac{\Theta_0(|z|)}{\Im(z)^k} \qquad\text{for all $z\in\bS_\rho$}$$ for some $k\geq 0$, and a sequence $(\epsilon_n)_{n=1}^\infty$ in $(0,\infty)$ converging to $0$ such that the estimate $$|G_{\mu_n}(z) - G_\nu(z)| \leq \Theta(z) \epsilon_n$$ holds for every $n\in\N$ and all $z\in \bS_\rho$.
 \item\label{it:cond-3} We have that $\sup_{n\in\N} \int_\R t^2\, d\mu_n(t) < \infty$.
\end{enumerate}
Then, $(\mu_n)_{n=1}^\infty$ converges in Kolmogorov distance to $\nu$; in fact, there is $D>0$, such that
$$\Delta(\mu_n,\nu) \leq D \epsilon_n^{\frac{\beta}{2+k+(2-\beta)l}} \qquad\text{for all $n\in\N$}.$$
\end{theorem}

The proof will be given in Subsection \ref{subsec:Hoelder_criterion_proof}. If one replaces \ref{it:cond-3} by the much stronger condition that all $\mu_n$ have support contained in a fixed compact interval, one can establish with similar but significantly simplified arguments a better rate for the Kolmogorov distance. We present the precise statement in the next theorem; a brief sketch of its largely straightforward proof is given at the end of Subsection \ref{subsec:Hoelder_criterion_proof}.

\begin{theorem}\label{thm:Hoelder_criterion_compact}
Let $(\mu_n)_{n=1}^\infty$ be a sequence of compactly supported Borel probability measures on $\R$ and let $\nu$ be any other Borel probability measure on $\R$. Suppose the following:
\begin{enumerate}
 \item The cumulative distribution function $\F_\nu$ of the measure $\nu$ is H\"older continuous with exponent $\beta\in(0,1]$ and a H\"older constant $C>0$.
 \item There are continuous functions $\Theta: \bS_\rho \to [0,\infty)$ and $\Theta_0: \overline{\bS_\rho} \to [0,\infty)$ for some $0<\rho\leq\infty$ that satisfy $$\Theta(z) \leq \frac{\Theta_0(z)}{\Im(z)^k} \qquad\text{for all $z\in\bS_\rho$}$$ for some $k\geq 0$, and a sequence $(\epsilon_n)_{n=1}^\infty$ in $(0,\infty)$ converging to $0$ such that the estimate $$|G_{\mu_n}(z) - G_\nu(z)| \leq \Theta(z) \epsilon_n$$ holds for every $n\in\N$ and all $z\in \bS_\rho$.
 \item There exists $M>0$ such that $\supp(\mu_n) \subseteq [-M,M]$ for all $n\in\N$.
\end{enumerate}
Then, $(\mu_n)_{n=1}^\infty$ converges in Kolmogorov distance to $\nu$; in fact, there is $D>0$, such that
$$\Delta(\mu_n,\nu) \leq D \epsilon_n^{\frac{\beta}{k+\beta}} \qquad\text{for all $n\in\N$}.$$
\end{theorem}

\subsection{Bai's inequalities}

The proof of Theorem \ref{thm:Hoelder_criterion} relies crucially on the following result, which is \cite[Theorem 2.2]{Bai93a}; see also \cite[Theorem 2.2]{Bai93b}.

\begin{theorem}\label{thm:Kolmogorov}
Let $\mu$ and $\nu$ be two Borel probability measures such that
\begin{equation}\label{eq:BS_condition-0}
\int_\R |\F_\mu(t) - \F_\nu(t)|\, dt < \infty.
\end{equation}
Then, for every $y>0$,
\begin{align*}
\Delta(\mu,\nu)
&\leq \frac{1}{\pi(1-\kappa)(2\gamma-1)}\Bigg[\int^A_{-A} |G_\mu(x+iy)-G_\nu(x+iy)|\, dx\\
&\quad + \frac{2\pi}{y} \int_{|t|>B} |\F_\mu(t) - \F_\nu(t)|\, dt + \frac{1}{y} \sup_{t\in\R} \int_{|s|\leq 2ya} |\F_\nu(t+s)-\F_\nu(t)|\, ds\Bigg],
\end{align*}
where $a$ and $\gamma$ are constants related to each other by
\begin{equation}\label{eq:BS_condition-1}
\gamma = \frac{1}{\pi} \int_{|x|<a} \frac{1}{x^2+1}\, dx > \frac{1}{2}
\end{equation}
and $A$, $B$, and $\kappa$ are positive constants such that $A>B$ and
\begin{equation}\label{eq:BS_condition-2}
\kappa = \frac{4B}{\pi(A-B)(2\gamma-1)} < 1.
\end{equation}
\end{theorem}

This useful methodology to control the Kolmogorov distance in terms of the corresponding Cauchy transforms is surveyed  nicely in the book \cite{BS10}.

\subsection{Bounding integrals of Cauchy transforms}

In order to apply Theorem \ref{thm:Kolmogorov}, we will have to control integrals of the form
$$\int_{|x|\geq A} |G_\mu(x+iy)-G_\nu(x+iy)|\, dx$$
as $A\to\infty$, uniformly over a large class of measures. Providing such bounds is the purpose of this subsection.

For a Borel probability measure $\mu$ on $\R$ having finite first and second moments, we denote by
$$m(\mu) := \int_\R t\, d\mu(t) \qquad\text{and}\qquad \sigma^2(\mu) := \int_\R (t-m(\mu))^2\, d\mu(t)$$
its \emph{mean} and \emph{variance}, respectively. Furthermore, in preparation of the next lemma, we define another quantity that is associated to $\mu$ and any real number $y>0$ by
$$W_y(\mu) := \bigg(1 + \frac{1}{2y} \int_\R |t|\, d\mu(t) + \frac{1}{2y^2} \int_\R t^2\, d\mu(t)\bigg)^{1/2}.$$
Moreover, if two such measures $\mu$ and $\nu$ are given, we put
$$c(\mu,\nu) := \Big(\sigma^2(\mu) + \sigma^2(\nu) + (m(\mu)-m(\nu))^2\Big)^{1/2}.$$
Using that notation, we are ready to formulate with the next lemma the desired integral bounds.

\begin{lemma}\label{lem:integral_bound}
Let $\mu$ and $\nu$ be any two Borel probability measures on $\R$ having finite first and second moments. Then, for each $y>0$ and for all $A>0$, it holds true that
\begin{equation}\label{eq:integral_bound-2}
\int_{|x|\geq A} |G_\mu(x+iy)-G_\nu(x+iy)|\, dx \leq c(\mu,\nu) W_y(\mu) W_y(\nu) \int_{|x|\geq A} \frac{1}{x^2+y^2}\, dx.
\end{equation}
\end{lemma}

The proof of Lemma \ref{lem:integral_bound} relies substantially on the following observation.

\begin{lemma}\label{lem:integral_bound-preliminary}
Let $\mu$ be a Borel probability measures on $\R$ having finite first and second moments. Then, for each $y>0$ and for all $A>0$,
\begin{equation}\label{eq:integral_bound-7}
\int_{|x|\geq A} \int_\R \frac{1}{(x-t)^2+y^2}\, d\mu(t)\, dx \leq W_y(\mu)^2 \int_{|x|\geq A} \frac{1}{x^2+y^2}\, dx.
\end{equation} 
\end{lemma}

\begin{proof}
Using Fubini's theorem and in turn a substitution, we may compute that
\begin{equation}\label{eq:integral_bound-5}
\begin{aligned}
\int_{|x|\geq A} \int_\R \frac{1}{(x-t)^2+y^2}\, d\mu(t)\, dx
&= \int_\R \int_{|x+t|\geq A} \frac{1}{x^2+y^2}\, dx\, d\mu(t).
\end{aligned}
\end{equation}
We want to control the integrand $\int_{|x+t|\geq A} \frac{1}{x^2+y^2}\, dx$ for every fixed $t\in\R$. We consider the case $t\geq0$ first. To begin with, we observe that
$$\{ x\in\R \mid |x+t|\geq A \} \cup (-A-t,-A] = \{ x\in\R \mid |x|\geq A\} \cup [A-t,A),$$
where, in the case $t<2A$, the sets on both sides are disjoint, and otherwise
\begin{align*}
\{ x\in\R \mid |x+t|\geq A \} \cap (-A-t,-A] &= \{ x\in\R \mid A-t \leq x \leq -A\}\\
                                             &= \{ x\in\R \mid |x|\geq A\} \cap [A-t,A).
\end{align*}
Thus, with respect to the measure $\rho_y$ that is given by $d\rho_y(x) = \frac{1}{x^2+y^2}\, dx$, we have in either case that
\begin{align*}
\lefteqn{\rho_y\big(\{ x\in\R \mid |x+t|\geq A \}\big) + \rho_y\big((-A-t,-A]\big)}\\
&\qquad = \rho_y\big(\{ x\in\R \mid |x|\geq A\}\big) + \rho_y\big([A-t,A)\big),
\end{align*}
which gives us that
\begin{align*}
\int_{|x+t|\geq A} \frac{1}{x^2+y^2}\, dx
&= \int_{|x|\geq A} \frac{1}{x^2+y^2}\, dx + \frac{t}{y} \int^{A+t}_A \frac{2(x-t)y}{(x-t)^2+y^2} \frac{1}{x^2+y^2}\, dx\\
&\qquad\qquad + t^2 \int^{A+t}_A \frac{1}{(x-t)^2+y^2} \frac{1}{x^2+y^2}\, dx.
\end{align*}
The second integral in the last line above can be estimated by the inequality of arithmetic and geometric means as
$$\bigg| \int^{A+t}_A \frac{2(x-t)y}{(x-t)^2+y^2} \frac{1}{x^2+y^2}\, dx \bigg|
 \leq \int^\infty_A \frac{1}{x^2+y^2}\, dx.$$
For the third integral, which has a positive integrand, we see that
$$\int^{A+t}_A \frac{1}{(x-t)^2+y^2} \frac{1}{x^2+y^2}\, dx \leq \frac{1}{y^2} \int^\infty_A \frac{1}{x^2+y^2}\, dx.$$
Thus, in summary, we have that
\begin{equation}\label{eq:integral_bound-6}
\int_{|x+t|\geq A} \frac{1}{x^2+y^2}\, dx \leq \Big(1 + \frac{|t|}{2y} + \frac{t^2}{2y^2}\Big) \int_{|x|\geq A} \frac{1}{x^2+y^2}\, dx.
\end{equation}
So far, we have established \eqref{eq:integral_bound-6} only in the case $t\geq0$, we claim, however, that it also holds for every $t\leq 0$. To see this, we note that the integral on the left hand side is taken over a mirror symmetric function, which gives that $\int_{|x+t|\geq A} \frac{1}{x^2+y^2}\, dx = \int_{|x+(-t)|\geq A} \frac{1}{x^2+y^2}\, dx$, and since the right hand side of \eqref{eq:integral_bound-6} remains the same if $t$ is replaced by $-t$, we infer that \eqref{eq:integral_bound-6} holds verbatim also for $t\leq 0$.

Inserting the bound \eqref{eq:integral_bound-6} into the formula \eqref{eq:integral_bound-5}, we obtain \eqref{eq:integral_bound-7}.
\end{proof}

In the following, the integral that appears on the left hand side of the inequality \eqref{eq:integral_bound-7} will be denoted by $J_\mu(y;A)$ .

\begin{proof}[Proof of Lemma \ref{lem:integral_bound}]
Let us first take any $z\in\C^+$. We may write
\[
G_\mu(z) - G_\nu(z) = \int_\R \int_\R \frac{t-s}{(z-t)(z-s)}\, d\mu(t)\, d\nu(s),
\]
which yields after an application of the Cauchy-Schwarz inequality
\begin{equation}\label{eq:integral_bound-3}
\begin{aligned}
\lefteqn{|G_\mu(z) - G_\nu(z)|}\\
& \qquad \leq c(\mu,\nu) \bigg(\int_\R \frac{1}{|z-t|^2}\, d\mu(t)\bigg)^{1/2} \bigg(\int_\R \frac{1}{|z-s|^2}\, d\nu(s)\bigg)^{1/2}.
\end{aligned}
\end{equation}
Now, let us fix any $y>0$. In order to establish \eqref{eq:integral_bound-2}, we use \eqref{eq:integral_bound-3} and again the Cauchy-Schwarz inequality; this gives for every $A>0$
\begin{equation}\label{eq:integral_bound-4} 
\int_{|x|\geq A} |G_\mu(x+iy) - G_\nu(x+iy)|\, dx \leq c(\mu,\nu) J_\mu(y;A)^{1/2} J_\nu(y;A)^{1/2}.
\end{equation}
Using Lemma \ref{lem:integral_bound-preliminary}, we infer from \eqref{eq:integral_bound-4} the validity of \eqref{eq:integral_bound-2}.
\end{proof}

\begin{remark}
Another interesting estimating which is however not sufficient for our purposes is the following:
\begin{equation}\label{eq:integral_bound-1}
\int_\R |G_\mu(x+iy)-G_\nu(x+iy)|\, dx \leq \frac{\pi}{y} c(\mu,\nu).
\end{equation}
It can be simply proved following the strategy of the proof of Lemma \ref{lem:integral_bound}.
\end{remark}

\subsection{Convergence in distribution and absolute moments}

Let us remind ourselves of the following well-known fact.

\begin{lemma}\label{lem:absolute_moments}
Let $(\mu_n)_{n=1}^\infty$ a sequence of Borel probability measures on $\R$ which converges in distribution to a Borel probability measure $\nu$ on $\R$. Suppose that, for some $p\geq 1$,
$$\sup_{n\in\N} \int_\R |t|^p\, d\mu_n(t) < \infty$$
holds. Then
$$\int_\R |t|^p\, d\nu(t) \leq \sup_{n\in\N} \int_\R |t|^p\, d\mu_n(t).$$
\end{lemma}

\subsection{The proof of Theorem \ref{thm:Hoelder_criterion}}\label{subsec:Hoelder_criterion_proof}

Now, we are prepared to give the proof of Theorem \ref{thm:Hoelder_criterion}. In doing so, we will follow the strategy of Theorem \ref{thm:Kolmogorov}, for which we will need the bounds that were derived in Lemma \ref{lem:integral_bound}.

\begin{proof}[Proof of Theorem \ref{thm:Hoelder_criterion}]
First, we fix $a$ and $\gamma$ according to the condition \eqref{eq:BS_condition-1} in Theorem \ref{thm:Kolmogorov} and we choose any $\kappa \in(0,1)$. We then define sequences $(y_n)_{n=1}^\infty$ and $(K_n)_{n=1}^\infty$ in $(0,\infty)$ by
$$y_n := \epsilon_n^{\frac{1}{2+k+(2-\beta)l}} \qquad\text{and}\qquad K_n := \frac{1}{y_n^{2-\beta}}$$
for every $n\in\N$; note that we clearly have $y_n \to 0$ and $K_n \to \infty$ as $n\to\infty$. We proceed now as follows:

\begin{itemize}
 \item The H\"older continuity condition in Item \ref{it:cond-1} yields for every $n\in\N$ that
$$\int_{|s|\leq 2y_na} |\F_\nu(t+s)-\F_\nu(t)|\, ds \leq 2C \int^{2y_na}_0 s^\beta\, ds = \frac{2C(2y_na)^{1+\beta}}{1+\beta}$$
and therefore, with $C_1 := \frac{2C(2a)^{1+\beta}}{1+\beta} > 0$, that
\begin{equation}\label{eq:Kolmogorov_cond-1}
\frac{1}{y_n} \sup_{t\in\R} \int_{|s|\leq 2y_na} |\F_\nu(t+s)-\F_\nu(t)|\, ds \leq C_1 y_n^\beta = C_1 \epsilon_n^{\frac{\beta}{2+k+(2-\beta)l}}.
\end{equation}

\item The condition formulated in Item \ref{it:cond-3} of the theorem guarantees that there are $m_1,m_2>0$ such that
$$\sup_{n\in\N} \int_\R |t|\, d\mu_n(t) \leq m_1 \qquad\text{and}\qquad \sup_{n\in\N} \int_\R t^2\, d\mu_n(t) \leq m_2.$$
Since the assumption made in Item \ref{it:cond-2} of the theorem guarantees due to Theorem \ref{thm:convergence_in_distribution} that $\mu_n \to \nu$ in distribution as $n\to\infty$, Lemma \ref{lem:absolute_moments} tells us that both
\begin{equation}\label{eq:absolute_moments}
\int_\R |t|\, d\nu(t) \leq m_1 \qquad\text{and}\qquad \int_\R t^2\, d\nu(t) \leq m_2.
\end{equation}
Consequently, we also have that
$$c := \sup_{n\in\N} c(\mu_n,\nu) < \infty$$
Using Lemma \ref{lem:integral_bound}, we get that
\begin{align*}
\lefteqn{\int_{|x|\geq K_n} |G_{\mu_n}(x+iy_n)-G_\nu(x+iy_n)|\, dx}\\
&\qquad \leq c W_{y_n}(\mu_n) W_{y_n}(\nu) \int_{|x|\geq K_n} \frac{1}{x^2+y_n^2}\, dx.
\end{align*}
We have then for every $n\in\N$
$$\int_{|x|\geq K_n} \frac{1}{x^2+y_n^2}\, dx \leq \int_{|x|\geq K_n} \frac{1}{x^2}\, dx = \frac{2}{K_n}$$
and furthermore, if $n$ is large enough, 
$$W_{y_n}(\mu_n) W_{y_n}(\nu) \leq \frac{m_2}{y_n^2}.$$
In combination, this shows that for sufficiently large $n\in\N$
$$\int_{|x|\geq K_n} |G_{\mu_n}(x+iy_n)-G_\nu(x+iy_n)|\, dx \leq \frac{2cm_2}{K_ny_n^2} = 2cm_2 \epsilon_n^{\frac{\beta}{2+k+(2-\beta)l}}.$$
We conclude that, for some suitable constant $C_2>0$, for all $n\in\N$
\begin{equation}\label{eq:Kolmogorov_cond-2.1}
\int_{|x|\geq K_n} |G_{\mu_n}(x+iy_n)-G_\nu(x+iy_n)|\, dx \leq C_2 \epsilon_n^{\frac{\beta}{2+k+(2-\beta)l}}.
\end{equation}

 \item Now, we invoke the estimates given in Item \ref{it:cond-2}. We put $R_n := (K_n^2+y_n^2)^{1/2}$ and we note first that for all sufficiently large $n\in\N$
 \begin{itemize}
  \item $R_n < 2^{1/l} K_n$,
  \item $\{x + iy_n \mid x\in [-K_n,K_n]\} \subset \bS_\rho$,
  \item $\max_{r\in[0,R_n]} \Theta_0(r) \leq \theta R_n^l$ for some $\theta>0$.
 \end{itemize}
 Thus, the bound on $\Theta$ yields that 
 $$\max_{x\in [-K_n,K_n]} \Theta(x+iy_n) \leq \frac{1}{y_n^k} \max_{r\in[0,R_n]} \Theta_0(r) \leq \theta \frac{R_n^l}{y_n^k} \leq 2 \theta \frac{K_n^l}{y_n^k},$$
 and with the bound for the Cauchy transforms we conclude that
 $$\max_{x\in[-K_n,K_n]} |G_{\mu_n}(x+iy_n)-G_\nu(x+iy_n)| \leq 2\theta \epsilon_n \frac{K_n^l}{y_n^k}.$$
Using this, we can now verify that for all such $n\in\N$
\begin{align*}
\lefteqn{\int^{K_n}_{-K_n} |G_{\mu_n}(x+iy_n)-G_\nu(x+iy_n)|\, dx}\\
& \qquad \leq 2 K_n \max_{x\in[-K_n,K_n]} |G_{\mu_n}(x+iy_n)-G_\nu(x+iy_n)|\\
& \qquad \leq 4 \theta \epsilon_n \frac{K_n^{l+1}}{y_n^k} = 4 \theta y_n^\beta = 4 \theta \epsilon_n^{\frac{\beta}{2+k+(2-\beta)l}}.
\end{align*}
Hence, we conclude that for all $n\in\N$, with some suitably chosen constant $C_3>0$,
\begin{equation}\label{eq:Kolmogorov_cond-2.2}
\int^{K_n}_{-K_n} |G_{\mu_n}(x+iy_n)-G_\nu(x+iy_n)|\, dx \leq C_3 \epsilon_n^{\frac{\beta}{2+k+(2-\beta)l}}.
\end{equation}

 \item By the fact that $\int_\R t^2\, d\mu_n(t) < \infty$ for every $n\in\N$, \eqref{eq:absolute_moments}, and the Chebyshev inequality, we get for every $n\in\N$
$$\int_\R |\F_{\mu_n}(t) - \F_\nu(t)|\, dt < \infty,$$
so that $\mu_n$ and $\nu$ satisfy condition \eqref{eq:BS_condition-0} of Theorem \ref{thm:Kolmogorov}; furthermore, this guarantees that we can choose $B_n>0$ such that
\begin{equation}\label{eq:Kolmogorov_cond-3}
\frac{1}{y_n} \int_{|t|>B_n} |\F_{\mu_n}(t) - \F_\nu(t)|\, dt \leq \epsilon_n^{\frac{\beta}{2+k+(2-\beta)l}}.
\end{equation}

 \item Now, we associate to the so found sequence $(B_n)_{n=1}^\infty$ another sequence $(A_n)_{n=1}^\infty$ by
$$A_n := B_n \bigg(1 + \frac{4}{\kappa\pi(2\gamma-1)}\bigg) \qquad\text{for all $n\in\N$}.$$
Then, for each $n\in\N$, we have that $A_n > B_n$ and \eqref{eq:BS_condition-2} is satisfied with the $\kappa$ that we have chosen above.

 \item Recall that by construction $A_n > K_n$ for every $n\in\N$; thus, we combine \eqref{eq:Kolmogorov_cond-2.1} and \eqref{eq:Kolmogorov_cond-2.2} to get for all $n\in\N$ that
\begin{equation}\label{eq:Kolmogorov_cond-2}
\int^{A_n}_{-A_n} |G_{\mu_n}(x+iy_n)-G_\nu(x+iy_n)|\, dx < (C_2 + C_3) \epsilon_n^{\frac{\beta}{2+k+(2-\beta)l}}.
\end{equation}

\end{itemize}

Putting these pieces together, we see that for every $n\in\N$, the conditions \eqref{eq:BS_condition-0}, \eqref{eq:BS_condition-1}, and \eqref{eq:BS_condition-2} are satisfied for $A_n$ and $B_n$; therefore, we may apply Theorem \ref{thm:Kolmogorov}, which yields, in combination with \eqref{eq:Kolmogorov_cond-1}, \eqref{eq:Kolmogorov_cond-2}, and \eqref{eq:Kolmogorov_cond-3}, that
$$\Delta(\mu_n,\nu) < D \epsilon_n^{\frac{\beta}{2+k+(2-\beta)l}} \qquad\text{with}\qquad D := \frac{1+C_1+C_2+C_3}{\pi(1-\kappa)(2\gamma-1)}$$
for all $n\in\N$, as claimed.
\end{proof}

\begin{proof}[Proof of Theorem \ref{thm:Hoelder_criterion_compact}]
Like in the proof of Theorem \ref{thm:Hoelder_criterion}, the asserted bound is obtained with the help of Theorem \ref{thm:Kolmogorov}. Here, we choose $A>B>M$ such that condition \eqref{eq:BS_condition-2} is satisfied, and apply Theorem \ref{thm:Kolmogorov} for $y_n := \epsilon_n^{\frac{1}{k+\beta}}$. The details are left to the reader.
\end{proof}

\section{Random matrix applications}\label{sec:random_matrices}

The aim of this section is to discuss some applications of our results in the context of random matrix theory. The simple idea is roughly the following: let $(X_1^{(N)},\dots,X_n^{(N)})$, for every $N\in\N$, be a tuple of selfadjoint random matrices of size $N\times N$ and suppose that their asymptotic behavior as $N\to\infty$ is described by a tuple $(X_1,\dots,X_n)$ of selfadjoint noncommutative random variables   in some tracial $W^\ast$-probability space $(\M,\tau)$ with the property that $\Phi^\ast(X_1,\dots,X_n) < \infty$. For many types of ``noncommutative functions'' $f$, the limiting eigenvalue distribution of the random matrices $Y^{(N)}=f(X_1^{(N)},\dots,X_n^{(N)})$ as $N\to\infty$ is given by the analytic distribution of the operator $Y=f(X_1,\dots,X_n)$.
We shall see how our results in Theorems \ref{thm:Hoelder_continuity}, \ref{thm:convergence_in_distribution}, and \ref{thm:Hoelder_criterion} could be combined to obtain H\"older continuity and provide rates of convergence with respect to the Kolmogorov distance for such matrix models.

As concrete instances of such ``composed'' random matrices we will consider here
\begin{itemize}
 \item for fixed (deterministic) selfadjoint matrices $a_0,a_1,\dots,a_n\in M_d(\C)$, the \emph{generalized block matrices}
\begin{equation}\label{eq:block_matrix} 
 Y^{(N)} := a_0 \otimes 1_N + \sum^n_{j=1} a_j \otimes X_j^{(N)};
\end{equation}
 \item for any selfadjoint noncommutative polynomial $P\in\C\langle x_1,\dots,x_n\rangle$ which is non-constant, the random matrices
\begin{equation}\label{eq:polynomial_evaluation} 
Y^{(N)} := P(X^{(N)}_1,\dots,X^{(N)}_n).
\end{equation}
\end{itemize}
In Section \ref{subsec:Gibbs_laws}, we will work with tuples $(X_1^{(N)},\dots,X_n^{(N)})$ of random matrices that follow general Gibbs laws; this includes the important case of GUEs, which is addressed separately in Section \ref{subsec:GUEs}. In Section \ref{subsec:random_matrices_basics}, we first recall some basic terminology.

\subsection{Random matrices and noncommutative probability theory}\label{subsec:random_matrices_basics}

Various types of random matrices fit nicely into the frame of noncommutative $\ast$-probability spaces. In fact, one can often treat them as noncommutative random variables in the $\ast$-probability space $(\M_N,\tau_N)$ given by the $\ast$-algebra $\M_N := M_N(\C) \otimes L^{\infty -}(\Omega,\bP)$ that is endowed with the tracial state $\tau_N := \tr_N \otimes \bE$ for some classical probability space $(\Omega,\F,\bP)$; note that $\bE$ stands for the associated expectation and $L^{\infty -}(\Omega,\bP) := \bigcap_{p\geq 1} L^p(\Omega,\bP)$.

Let a selfadjoint random matrix $X^{(N)} \in \M_N$ be given. We will be interested in the random eigenvalues $\lambda_1(X^{(N)}),\dots,\lambda_N(X^{(N)})$ of $X^{(N)}$, to which we associate a random probability measure $\mu_{X^{(N)}}$ on $\R$ by
$$\mu_{X^{(N)}} := \frac{1}{N} \sum_{j=1}^N \delta_{\lambda_j(X^{(N)})},$$
called the \emph{empirical eigenvalue distribution of $X^{(N)}$}. By $\overline{\mu}_X$, we will denote the \emph{mean eigenvalue distribution of $X^{(N)}$} which is the probability measure on $\R$ that is defined as $\overline{\mu}_X := \bE[\mu_X]$. We point out that the Cauchy transform of $\overline{\mu}_{X^{(N)}}$ agrees with the Cauchy transform of the noncommutative random variable $X^{(N)}$ in $(\M_N,\tau_N)$, i.e., we have
$$G_{\overline{\mu}_{X^{(N)}}}(z) = \tau_N\big((z 1_N - X^{(N)})^{-1}\big) \qquad\text{for all $z\in\C^+$}.$$

In the following, we shall see random matrices as elements in $M_N(\C)_\sa$ chosen randomly according to some probability measure on this space.

\subsection{Gibbs laws}\label{subsec:Gibbs_laws}

Consider a selfadjoint noncommutative polynomial $V\in\C\langle x_1,\dots,x_n\rangle$; in the following, we will refer to $V$ as a \emph{potential}. Following \cite{GS09}, we say that the potential $V$ is \emph{selfadjoint $(c,M)$-convex} if
$$(DV(X) - DV(Y)) . (X-Y) \geq c (X-Y) . (X-Y)$$
for any $n$-tuples $X=(X_1,\dots,X_n)$ and $Y=(Y_1,\dots,Y_n)$ of selfadjoint operators in some $C^\ast$-algebra $\A$ that are bounded in norm by $M$, where $X . Y := \frac{1}{2} \sum^n_{j=1} (X_j Y_j + Y_j X_j)$.

Suppose now that $V$ is selfadjoint $(c,\infty)$-convex for some $c>0$. We will use $V$ to introduce a probability measure on $M_N(\C)^n_\sa$. For that purpose, let us first define the Lebesgue measure on $M_N(\C)_\sa$ by
$$dX^{(N)} := \prod_{k=1}^N dX_{kk} \prod_{1\leq k < l \leq N} d\Re(X_{kl})\, d\Im(X_{kl}).$$
Further, let $\Tr$ denote the unnormalized trace on $M_N(\C)$. On the space $M_N(\C)^n_\sa$, we then define the probability measure $\bP_V^N$ by
$$d\bP_V^N(X_1^{(N)},\dots,X_n^{(N)}) = \frac{1}{Z_N(V)} \exp\big(-N \Tr(V(X^{(N)}_1,\dots,X^{(N)}_n))\big)\ dX_1^{(N)}\, \dots\, dX_n^{(N)},$$
where $Z_N(V)$ is the normalizing constant that is given by
$$Z_N(V) := \int_{M_N(\C)^n_\sa} \exp\big(-N \Tr(V(X^{(N)}_1,\dots,X^{(N)}_n))\big)\ dX_1^{(N)}\, \dots\, dX_n^{(N)}.$$
We call $\bP^N_V$ the \emph{Gibbs measure with potential $V$}.

The Brascamp-Lieb inequality \cite{BL76} guarantees that those measures are well-defined (i.e., that $Z_N(V)$ is finite) for potentials $V$ that are selfadjoint $(c,\infty)$-convex for some $c>0$. Those measures are extensively studied for instance in \cite{GMS06,GMS07,GS09}; see also the surveys \cite{Gui06,Gui14,Gui16}.

It was shown in \cite{GS09} that $n$-tuples $(X^{(N)}_1,\dots,X^{(N)}_n)$ of selfadjoint random matrices of size $N\times N$ following the Gibbs law $\bP^N_V$ can be described in the limit $N\to\infty$ by an $n$-tuple $(X_1,\dots,X_n)$ of selfadjoint operators in some tracial $W^\ast$-probability space with the property that $\Phi^\ast(X_1,\dots,X_n) <\infty$. Before we can state their result, we need to introduce some further notation: for every noncommutative polynomial $V\in\C\langle x_1,\dots,x_n\rangle$, we denote by $DV = (D_1 V, \dots, D_n V)$ the \emph{cyclic gradient} of $V$; the \emph{cyclic derivatives} $D_1 V, \dots, D_n V$ of $V$ are given by $D_j V = \tilde{m}(\partial_j V)$ for $j=1,\dots,n$, where $\tilde{m}: \C\langle x_1,\dots,x_n\rangle \otimes \C\langle x_1,\dots,x_n\rangle \to \C\langle x_1,\dots,x_n\rangle$ denotes the flipped multiplication that is determined by $\tilde{m}(P_1 \otimes P_2) := P_2 P_1$.

\begin{theorem}[{\cite[Theorem 1.6]{GS09}}]\label{thm:Gibbs}
Let $V$ be selfadjoint $(c,\infty)$-convex for some $c>0$. For every $N\in\N$, let $X^{(N)}=(X^{(N)}_1,\dots,X^{(N)}_n)$ be an $n$-tuple of selfadjoint random matrices of size $N\times N$ with law $\bP^N_V$. Then there is an $n$-tuple $X=(X_1,\dots,X_n)$ of selfadjoint operators in some tracial $W^\ast$-probability space $(\M,\tau)$ (whose joint distribution $\mu_X$ is then in fact uniquely determined) which satisfy the \emph{Schwinger-Dyson equation with respect to the potential $V$}, i.e.,
$$(\tau\otimes\tau)\big((\partial_j P)(X)\big) = \tau\big(P(X) (D_j V)(X)\big)$$
for every $P\in\C\langle x_1,\dots,x_n\rangle$ and all $j=1,\dots,n$, and it holds true for each $P\in\C\langle x_1,\dots,x_n\rangle$ that
$$\lim_{N\to\infty} \tr_N(P(X^{(N)})) = \tau(P(X)) \qquad\text{almost surely}.$$
\end{theorem}

In the situation of Theorem \ref{thm:Gibbs}, the Schwinger-Dyson equation yields that $(\xi_1,\dots,\xi_n)$ with $\xi_j := (D_j V)(X)$ for $j=1,\dots,n$ are the conjugate system for $X=(X_1,\dots,X_n)$; in fact, since $V$ is a polynomial, $(\xi_1,\dots,\xi_n)$ are Lipschitz conjugate variables for $X$.

With the result obtained in the previous subsection, we conclude the following about matrix models of the type \eqref{eq:polynomial_evaluation}.

\begin{corollary}\label{cor:random_matrices_polynomial}
In the situation of Theorem \ref{thm:Gibbs}, the following holds for each selfadjoint noncommutative polynomial $P\in\C\langle x_1,\dots,x_n\rangle$ of degree $d\geq 1$:
\begin{enumerate}
 \item The empirical eigenvalue distribution $\mu_{Y^{(N)}}$ of $$Y^{(N)} = P(X^{(N)}_1,\dots,X^{(N)}_n)$$ converges in distribution almost surely to a compactly supported Borel probability measure $\nu$ on $\R$ whose cumulative distribution function is H\"older continuous with exponent $\frac{1}{2^d-1}$.
 \item We have that $$\lim_{N\to\infty} \Delta(\mu_{Y^{(N)}},\nu) = 0 \quad \text{almost surely} \quad\text{and}\quad \lim_{N\to\infty} \Delta(\overline{\mu}_{Y^{(N)}},\nu) = 0.$$
\end{enumerate}
\end{corollary}

\begin{proof}
Theorem \ref{thm:Gibbs} tells us that $\mu_{Y^{(N)}}$ converges in distribution almost surely as $N\to\infty$ to the analytic distribution $\nu:=\mu_Y$ of $Y:=P(X_1,\dots,X_n)$. Since $X_1,\dots,X_n$ satisfy the Schwinger-Dyson equation with potential $V$, we infer that $\Phi^\ast(X_1,\dots,X_n)<\infty$ with Lipschitz conjugate variables as outlined above. Therefore, with the help of Theorem \ref{thm:Hoelder_continuity}, we see that the cumulative distribution function of $\nu$ is H\"older continuous with exponent $\frac{1}{2^d-1}$.

As a consequence of Theorem \ref{thm:convergence_in_distribution}, we obtain that $\Delta(\mu_{Y^{(N)}},\nu) \to 0$ almost surely and in particular $\Delta(\overline{\mu}_{Y^{(N)}},\nu) \leq \bE[\Delta(\mu_{Y^{(N)}},\nu)] \to 0$ as $N\to \infty$.
\end{proof}

We point out that an analogous statement holds true for certain random matrices of the form \eqref{eq:block_matrix}. For that purpose, we need the following terminology: if $a_1,\dots,a_n\in M_d(\C)$ are selfadjoint matrices, we call
\begin{equation}\label{eq:quantum_operator}
\cL:\ M_d(\C) \to M_d(\C),\quad b\mapsto \sum^n_{j=1} a_j b a_j
\end{equation}
the \emph{quantum operator (associated to $a_1,\dots,a_n$)}; we say that $\cL$ is \emph{semi-flat}, if there is some constant $c>0$ such that $\cL(b) \geq c \tr_d(b) 1_d$ for all positive semidefinite matrices $b\in M_d(\C)$.

In \cite[Theorem 8.1]{MSY18}, it is stated that whenever $a_0,a_1,\dots,a_n\in M_d(\C)$ are selfadjoint matrices such that the quantum operator $\cL: M_d(\C) \to M_d(\C)$ associated to $a_1,\dots,a_n$ is semi-flat and $X_1,\dots,X_n$ are selfadjoint operators in a tracial $W^\ast$-probability space $(\M,\tau)$ satisfying $\Phi^\ast(X_1,\dots,X_n) < \infty$, then $\F_Y$ is H\"older continuous with exponent $\beta=\frac{2}{3}$ for the selfadjoint operator in the tracial $W^\ast$-probability space $(M_d(\C) \otimes \M, \tr_d \otimes \tau)$ given by
$$Y:=a_0 \otimes 1 + \sum^n_{j=1} a_j \otimes X_j.$$
This approach was inspired by \cite{AjEK18,AEK18}, where a very detailed analysis of such operators in the special case for freely independent semicircular operators $X_1,\dots,X_n$ is carried out.

\begin{corollary}\label{cor:random_matrices_block}
In the situation of Theorem \ref{thm:Gibbs}, for every choice of selfadjoint matrices $a_0,a_1,\dots,a_n\in M_d(\C)$ for which the quantum operator $\cL: M_d(\C) \to M_d(\C)$ associated to $a_1,\dots,a_n$ is semi-flat, the following statements hold true:
\begin{enumerate}
 \item The empirical eigenvalue distribution $\mu_{Y^{(N)}}$ of the random matrix $$Y^{(N)} = a_0 \otimes 1_N + \sum^n_{j=1} a_j \otimes X^{(N)}_j$$ converges in distribution almost surely to a compactly supported Borel probability measure $\nu$ on $\R$ whose cumulative distribution function is H\"older continuous with exponent $\frac{2}{3}$.
 \item We have that $$\lim_{N\to\infty} \Delta(\mu_{Y^{(N)}},\nu) = 0 \quad \text{almost surely} \quad\text{and}\quad \lim_{N\to\infty} \Delta(\overline{\mu}_{Y^{(N)}},\nu) = 0.$$
\end{enumerate}
\end{corollary}

\begin{proof}
It follows from Theorem \ref{thm:Gibbs} that $\mu_{Y^{(N)}}$ converges in distribution almost surely as $N\to\infty$ to the analytic distribution $\nu:=\mu_Y$ of the operator $Y:=a_0 \otimes 1 + \sum^n_{j=1} a_j \otimes X_j$ living in $(M_d(\C) \otimes \M, \tr_d \otimes \tau)$. Since $X_1\dots,X_n$ satisfy $\Phi^\ast(X_1,\dots,X_n) < \infty$, we can use \cite[Theorem 8.1]{MSY18} which tells us that the cumulative distribution function of $\nu$ is H\"older continuous with exponent $\frac{2}{3}$. The rest is shown like in the proof of Corollary \ref{cor:random_matrices_polynomial}.
\end{proof}

\subsection{Gaussian random matrices and rates of convergence}\label{subsec:GUEs}

A \emph{(standard) selfadjoint Gaussian random matrix} (or \emph{GUE}) of size $N\times N$ is a selfadjoint complex random matrix $X=(X_{kl})_{k,l=1}^N$ in $\M_N$ for which
$$\{X_{kk} \mid 1\leq k \leq n\} \cup \{\Re(X_{kl}) \mid 1\leq k < l \leq N\} \cup \{\Im(X_{kl}) \mid 1\leq k < l \leq N\}$$
are independent real Gaussian random variables such that
$$\bE[X_{kl}] = 0 \quad\text{and}\quad \bE[|X_{kl}|^2] = \frac{1}{N} \quad \text{for $1\leq k \leq l \leq N$}.$$
Those fall into the general class of Gibbs measures considered in the previous section with the particular potential  $V = \frac{1}{2}(x_1^2 + \dots + x_n^2)$.

Our goal is to strengthen Corollaries \ref{cor:random_matrices_polynomial} and \ref{cor:random_matrices_block} in the GUE case by proving explicit rates for the Kolmogorov distance. This improvement depends crucially on the results of \cite{HT05} about random matrices of the form \eqref{eq:block_matrix}, which we are going to recall now.

Note that each random matrix like in \eqref{eq:block_matrix} is an element in $\M_{dN} \cong M_d(\C) \otimes \M_N$. For each $X = X^\ast \in M_d(\C) \otimes \M_N$, we define its \emph{matrix-valued Cauchy transform} by
$$\bG_X:\ \bH^+(M_d(\C)) \to \bH^-(M_d(\C)), \quad b \mapsto (\id_{M_d(\C)} \otimes \tau_N)\big((b \otimes 1_N - X)^{-1}\big),$$
where $\bH^+(M_d(\C))$ and $\bH^-(M_d(\C))$ denote the upper and lower half-plane in $M_d(\C)$, respectively, that is, the set of all $b\in M_d(\C)$ with positive and negative imaginary part $\Im(b) := \frac{1}{2i}(b-b^\ast)$, respectively. Note that $G_{\overline{\mu}_{X}}(z) = \tr_d(\bG_{X}(z 1_d))$ for all $z\in\C^+$. 

The limit of those random matrices will be described accordingly by some selfadjoint operator in the tracial $W^\ast$-probability space $(M_d(\C) \otimes \M, \tr_d \otimes \tau)$. Note that $(M_d(\C) \otimes \M, \tr_d \otimes \tau)$ is again a tracial $W^\ast$-probability space, which can further be regarded as an operator-valued probability spaces over $M_d(\C)$ with the conditional expectation that is given by $\id_{M_d(\C)} \otimes \tau$. Accordingly, we can consider the matrix-valued Cauchy transform of any $X = X^\ast \in M_d(\C) \otimes \M$; it is defined by
$$\bG_X:\ \bH^+(M_d(\C)) \to \bH^-(M_d(\C)), \quad b \mapsto (\id_{M_d(\C)} \otimes \tau)\big((b \otimes 1 - X)^{-1}\big).$$

Now, we can formulate the precise convergence result, which is \cite[Theorem 5.7]{HT05}.

\begin{theorem}\label{thm:HT}
Let $a_0,a_1,\dots,a_n \in M_d(\C)$ be selfadjoint matrices. We consider, for each $N\in\N$, a tuple $(X^{(N)}_1,\dots,X^{(N)}_n)$ of $n$ independent GUEs. Let further $(S_1,\dots,S_n)$ be a tuple of freely independent semicircular elements in some tracial $W^\ast$-probability space $(\M,\tau)$. Consider
$$X^{(N)} := a_0 \otimes 1_N + \sum^n_{j=1} a_j \otimes X_j^{(N)} \qquad\text{and}\qquad S := a_0 \otimes 1 + \sum^n_{j=1} a_j \otimes S_j.$$
Then the matrix-valued Cauchy transforms $\bG_{X^{(N)}}, \bG_S: \bH^+(M_d(\C)) \to \bH^-(M_d(\C))$ satisfy
$$\|\bG_{X^{(N)}}(b) - \bG_S(b)\| \leq \frac{4C}{N^2} (K + \|b\|)^2 \|\Im(b)^{-1}\|^7$$
for all $b\in\bH^+(M_d(\C))$, with the constants $C>0$ and $K>0$ that are given by
$$C = d^3 \bigg\|\sum^n_{j=1} a_j^2\bigg\|^2 \qquad\text{and}\qquad K = \|a_0\| + 4 \sum^n_{j=1} \|a_j\|.$$
\end{theorem}

Accordingly (see \cite[Lemma 6.1]{HT05}), the Cauchy transforms $G_{\overline{\mu}_{X^{(N)}}}$ and $G_{\mu_S}$, which are related to the respective matrix-valued Cauchy transforms by $G_{\overline{\mu}_{X^{(N)}}}(z) = \tr_d(\bG_{X^{(N)}}(z 1_d))$ and $G_{\mu_S}(z) = \tr_d(\bG_S(z 1_d))$ for every $z\in\C^+$, satisfy
\begin{equation}\label{eq:HT}
|G_{\overline{\mu}_{X^{(N)}}}(z) - G_{\mu_S}(z)| \leq \frac{4C}{N^2} \frac{(K + |z|)^2}{\Im(z)^7}.
\end{equation}

Putting these facts together, we conclude now the following.

\begin{corollary}\label{cor:block-GUE}
Let $a_0,a_1,\dots,a_n \in M_d(\C)$ be selfadjoint such that the quantum operator $\cL: M_d(\C) \to M_d(\C)$ associated to $a_1,\dots,a_n$ by \eqref{eq:quantum_operator} is semi-flat.
For each $N\in\N$, let $(X^{(N)}_1,\dots,X^{(N)}_n)$ be a tuple of $n$ independent GUEs. Further, let $(S_1,\dots,S_n)$ be a tuple of freely independent semicircular elements in some tracial $W^\ast$-probability space $(\M,\tau)$. Set
$$X^{(N)} := a_0 \otimes 1_N + \sum^n_{j=1} a_j \otimes X_j^{(N)} \qquad\text{and}\qquad S := a_0 \otimes 1 + \sum^n_{j=1} a_j \otimes S_j.$$
Then the averaged empirical eigenvalue distribution $\overline{\mu}_{X^{(N)}}$ of $X^{(N)}$ satisfies
$$\Delta(\overline{\mu}_{X^{(N)}},\mu_S) \leq D N^{-4/35}.$$
\end{corollary}

\begin{proof}
We want to apply Theorem \ref{thm:Hoelder_criterion}. Therefore, we check that $\mu_N := \overline{\mu}_{X^{(N)}}$ and $\nu := \mu _S$ have the required properties:
\begin{itemize}
 \item Since $\cL$ is semi-flat, the cumulative distribution function of $\mu_S$ is H\"older continuous with exponent $\beta=\frac{2}{3}$, as it follows from \cite[Theorem 8.1]{MSY18}. 
 \item Let us define $\epsilon_N := N^{-2}$. Then, due to \eqref{eq:HT}, we have that
$$|G_{\overline{\mu}_{X^{(N)}}}(z) - G_{\mu_S}(z)| \leq \Theta(z) \epsilon_N \qquad\text{for all $z\in\C^+$}$$
with a continuous function $\Theta: \C^+ \to [0,\infty)$ that satisfies the growth condition $\Theta(z) \leq \frac{\Theta_0(|z|)}{\Im(z)^7}$ on $\bS_\infty = \C^+$ with the continuous function $\Theta_0: [0,\infty) \to [0,\infty)$ that is given by $\Theta_0(r) := (K+r)^2$; the latter satisfies $\lim_{R\to\infty} R^{-2} \max_{r\in[0,R]} \Theta_0(r) = 1$. 
 \item For each $N\in\N$, the measure $\overline{\mu}_{X^{(N)}}$ satisfies $$\int_\R t^2\, d\overline{\mu}_{X^{(N)}}(t) = \bE\big[(\tr_d\otimes\tr_N)\big( (X^{(N)})^2 \big)\big] = \tr_d(\cL(1_d)).$$ 
\end{itemize}
Therefore, Theorem \ref{thm:Hoelder_criterion} guarantees the existence of some numerical constant $D>0$ for which $\Delta(\overline{\mu}_{X^{(N)}},\mu_S) \leq D N^{-4/35}$ holds, as claimed.
\end{proof}

\begin{remark}
In the proof of Theorem \ref{thm:Hoelder_criterion}, on which the previous corollary relies substantially, the behavior of the cumulative distribution functions near $\infty$ was controlled with the help of Chebyshev's inequality. For the sake of completeness, we note that in the case of the mean empirical eigenvalue distribution $\overline{\mu}_{X^{(N)}}$ of $X^{(N)}$ much stronger statements are possible -- although this does not improve the conclusion of Theorem \ref{thm:Hoelder_criterion}. More precisely, we have
\begin{equation}\label{eq:cumulative_decay}
1-\F_{\overline{\mu}_{X^{(N)}}}(2+t) \leq  2N \exp\Big(-\frac{N t^2}{2}\Big) \quad\text{and}\quad \F_{\overline{\mu}_{X^{(N)}}}(2-t) \leq 2N \exp\Big(-\frac{N t^2}{2}\Big).
\end{equation}
This follows from \cite[Proof of Lemma 3.3]{HT03}, \cite[Proof of Lemma 6.4]{S05}, and \cite[Proof of Proposition 6.4]{HST06}.
\end{remark}

With the help of linearization techniques that we outline in Section \ref{sec:linearization} of the appendix, we can give rates for the Kolmogorov distance also in the case of polynomial evaluations. 

\begin{corollary}\label{cor:p-GUE}
Let $P\in\C\langle x_1,\dots,x_n\rangle$ be a selfadjoint noncommutative polynomial of degree $d\geq 1$. For each $N\in\N$, we consider a tuple $X^{(N)} = (X^{(N)}_1,\dots,X^{(N)}_n)$ of $n$ independent GUEs. Further, let $S=(S_1,\dots,S_n)$ be a tuple of freely independent semicircular elements in some tracial $W^\ast$-probability space $(\M,\tau)$. We define
$$Y^{(N)} := P(X^{(N)}_1,\dots,X^{(N)}_n) \qquad\text{and}\qquad Y := P(S_1,\dots,S_n).$$
Then there is a constant $D>0$ such that for all $N\in\N$
$$\Delta(\overline{\mu}_{Y^{(N)}},\mu_Y) \leq D N^{-\frac{1}{13 \cdot 2^{d+2} - 60}}.$$
\end{corollary}

For the particular case $P(x)=x$ of a GUE matrix, the rate of convergence to the semicircular distribution with respect to the Kolmogorov distance was studied by G\"otze and Tikhomirov in \cite{GT02} and then in \cite{GT05} where they obtain the optimal rate, conjectured by Bai \cite{Bai93a} for the more general Wigner matrices. Even for $d=1$ or $n=1$, our result still covers a larger class of matrices than a single GUE.

\begin{proof}[Proof of Corollary \ref{cor:p-GUE}]
The assertion will follow from Theorem \ref{thm:Hoelder_criterion}. Note that the convergence in distribution of $(\overline{\mu}_{Y^{(N)}})_{N=1}^\infty$ to $\mu_Y$ can be taken for granted as by the results of \cite{Voi91} on asymptotic freeness, the tuple $X^{(N)}$ is known to converge in distribution to $S$ as $N\to \infty$.

First of all, Theorem \ref{thm:Hoelder_continuity} yields that the cumulative distribution function of the analytic distribution of $Y$ is H\"older continuous with exponent $\frac{1}{2^d-1}$.

In order to verify condition \ref{it:cond-2} of Theorem \ref{thm:Hoelder_criterion}, we choose a selfadjoint linear representation $\rho=(u,Q,v)$ of $P$ and we consider the associated selfadjoint linearization $\hat{P}$.

For a moment, let us fix $z\in\C^+$ and $N\in\N$; we define $\epsilon>0$ by $\epsilon := N^{-1/4} \Im(z)$.
Since in particular $\epsilon \leq \Im(z) \leq |z|$, we see that $\|\Lambda_\epsilon(z)\| = |z|$ and $\|\Im(\Lambda_\epsilon(z))^{-1}\| = \frac{1}{\epsilon} = N^{1/4} \frac{1}{\Im(z)}$. Thus, involving Theorem \ref{thm:HT}, we get that
$$\|\bG_{\hat{P}(X^{(N)})}(\Lambda_\epsilon(z)) - \bG_{\hat{P}(S)}(\Lambda_\epsilon(z))\| \leq 4C N^{-1/4} \frac{(K + |z|)^2}{\Im(z)^7}.$$
Furthermore, by Theorem \ref{thm:approximation_improved}, we find noncommutative polynomials $P=P_1,P_2,\dots,P_d\in\C\langle x_1,\dots,x_n\rangle$ such that
\begin{align*}
\big| G_{P(X^{(N)})}(z) - \big[ \bG_{\hat{P}(X^{(N)})}(\Lambda_\epsilon(z)) \big]_{1,1} \big| &\leq N^{-1/4} \frac{2C'}{\Im(z)},\\
\big| G_{P(X)}(z) - \big[ \bG_{\hat{P}(X)}(\Lambda_\epsilon(z)) \big]_{1,1} \big| &\leq N^{-1/4} \frac{2C'}{\Im(z)},
\end{align*}
with the constant 
$$C' := \sum^d_{j=1} \|P_j(S)\|_2^2 + \sup_{N\in\N} \sum^d_{j=1} \bE\big[\tr_N\big(P_j(X^{(N)})^\ast P_j(X^{(N)})\big)\big];$$
note that $C'$ is finite because $X^{(N)}$ converges in distribution to $S$ as $N\to \infty$ and therefore
$$\lim_{N\to\infty} \sum^d_{j=1} \bE\big[\tr_N\big(P_j(X^{(N)})^\ast P_j(X^{(N)})\big)\big] = \sum^d_{j=1} \|P_j(S)\|_2^2.$$
Since $|[A]_{1,1}| \leq \|A\|$ for every matrix $A\in M_N(\C)$, we obtain by putting these pieces together that
\begin{equation}\label{eq:HST_variant}
|G_{\overline{\mu}_{Y^{(N)}}}(z) - G_{\mu_Y}(z)| \leq N^{-1/4} \Big(\frac{2C'}{\Im(z)} + 4C \frac{(K + |z|)^2}{\Im(z)^7}\Big)
\end{equation}

Thus, in summary, we see that with the continuous function
$$\Theta:\ \C^+ \to [0,\infty),\quad z\mapsto \frac{2C'}{\Im(z)} + 4C \frac{(K + |z|)^2}{\Im(z)^7}$$
we have for all $z\in\C^+$ and for all $N\in\N$ that
$$|G_{\overline{\mu}_{Y^{(N)}}}(z) - G_{\mu_Y}(z)| \leq \Theta(z) N^{-1/4}.$$
Taking now a closer look at $\Theta$, we see that it can be bounded on the strip $\bS_1$ as
$$\Theta(z) \leq \frac{\Theta_0(|z|)}{\Im(z)^7} \qquad\text{for all $z\in\bS_1$},$$
where the function $\Theta_0: [0,\infty) \to [0,\infty)$ is defined by $\Theta_0(r) := 2 C' + (K + r)^2$ and thus satisfies the growth condition $\limsup_{R\to\infty} R^{-2} \max_{r\in[0,R]} \Theta_0(r) = 1$.
This means that condition \ref{it:cond-1} of Theorem \ref{thm:Hoelder_criterion} is fulfilled with $l=2$, $k=7$, and the sequence $(\epsilon_N)_{N=1}^\infty$ defined by $\epsilon_N := N^{-1/4}$.

It remains to check condition \ref{it:cond-3} of Theorem \ref{thm:Hoelder_criterion}. This, however, is clear as
$$\lim_{N\to\infty} \int_\R t^2\, d\overline{\mu}_{Y^{(N)}} = \lim_{N\to\infty} \bE\big[\tr_N(P(X^{(N)})^2)\big] = \tau\big(P(S)^2\big),$$
since $X^{(N)}$ converges in distribution to $S$ as $N\to \infty$.

Thus, Theorem \ref{thm:Hoelder_criterion} guarantees the existence of a constant $D>0$ such that
$$\Delta(\overline{\mu}_{Y^{(N)}},\mu_Y) \leq D N^{-\frac{1}{13 \cdot 2^{d+2} - 60}} \qquad\text{for all $N\in\N$},$$
which proves the assertion.
\end{proof}

We point out that \eqref{eq:HST_variant} is a slightly improved variant of the related inequality (6.3) in \cite{HST06}. While the latter required an elaborate generalization of Theorem \ref{thm:HT} from \cite{HT05}, we can work with Theorem \ref{thm:HT} directly; this simplification is possible thanks to universal Theorem \ref{thm:approximation_improved} which we present in the appendix.

\appendix

\section{Approximation of Cauchy transforms by linearizations}\label{sec:linearization}

Linearization techniques have turned out to be very useful when dealing with evaluations of noncommutative polynomials or noncommutative rational functions; see, for instance, \cite{BMS17,HMS18} and the references collected therein.

Here, we focus on the case of scalar-valued noncommutative polynomials $P\in\C\langle x_1,\dots,x_n\rangle$. We can associate to $P$ by purely algebraic techniques a \emph{linear representation} $\rho=(u,Q,v)$, i.e., a triple that consists of a row vector $u$ and a column vector $v$, both of the size, say $d\in\N$, and an invertible matrix $Q \in M_d(\C\langle x_1,\dots,x_n\rangle)$ of the form
$$Q = Q_0 + Q_1 x_1 + \dots + Q_n x_n$$
with scalar matrices $Q_0,Q_1,\dots,Q_n\in M_d(\C)$ which enjoys the crucial property that $P = -u Q^{-1} v$. Moreover, if $P$ is selfadjoint, we may find a particular linear representation $\rho$ which is additionally \emph{selfadjoint} in the sense that $v=u^\ast$ holds and all matrices $Q_0,Q_1,\dots,Q_n$ are selfadjoint. The matrix-valued but linear polynomial
\begin{equation}\label{eq:linearization}
\hat{P} := \begin{bmatrix} 0 & u\\ v & Q\end{bmatrix} \in M_{d+1}(\C\langle x_1,\dots,x_n\rangle)
\end{equation}
is called the \emph{linearization of $P$ associated to $\rho$}; note that $\hat{P}$ is selfadjoint if and only if $\rho$ is selfadjoint.

Suppose now that $(\M,\tau)$ is a tracial $W^\ast$-probability space and let $X=(X_1,\dots,X_n)$ be any $n$-tuple of selfadjoint operators in $\M$. Further, let $\rho=(u,Q,v)$ be a selfadjoint linear representation of $P=P^\ast \in\C\langle x_1,\dots,x_n\rangle$ and let $\hat{P} \in M_{d+1}(\C\langle x_1,\dots,x_n\rangle)$ be the associated linearization.
With the help of the well-known Schur complement formula, one easily sees that the scalar-valued Cauchy transform of $P(X) = -u Q(X)^{-1} v$ can be obtained from the matrix-valued Cauchy transform of the selfadjoint operator $\hat{P}(X) \in M_{d+1}(\M)$; in fact, we have for every point $z\in\C^+$ that
\begin{equation}\label{eq:Cauchy-linearization}
G_{P(X)}(z) = \lim_{\epsilon \searrow 0} \big[\bG_{\hat{P}(X)}(\Lambda_\epsilon(z))\big]_{1,1},
\end{equation}
where $[A]_{1,1}:= A_{11}$ for any matrix $A$ with entries $A_{ij}$ and $\Lambda_\epsilon(z)$ is a matrix in $\bH^+(M_{d+1}(\C))$ that is given by
$$\Lambda_\epsilon(z) := \begin{bmatrix} z & 0 & \hdots & 0\\ 0 & i\epsilon & \hdots & 0\\ \vdots & \vdots & \ddots & \vdots\\ 0 & 0 & \hdots & i\epsilon\end{bmatrix}.$$
Notably, when the operators $X_1,\dots,X_n$ are freely independent, one can compute $\bG_{\hat{P}(X)}$ efficiently at any point in $\bH^+(M_{d+1}(\C))$ out of the individual analytic distributions of $X_1,\dots,X_n$ by means of operator-valued free probability theory; see \cite{BMS17}.

Our goal is the following quantitative version of \eqref{eq:Cauchy-linearization}.

\begin{theorem}\label{thm:approximation_improved}
Let $P\in\C\langle x_1,\dots,x_n\rangle$ be a selfadjoint noncommutative polynomial. Consider the selfadjoint linearization $\hat{P} \in M_{d+1}(\C\langle x_1,\dots,x_n\rangle)$ of $P$ associated to a given selfadjoint linear representation $\rho=(u,Q,v)$ of $P$ with $u\neq 0$. Then there are (not necessarily selfadjoint) polynomials $P_1,\dots,P_d\in\C\langle x_1,\dots,x_n\rangle$, where $P_1$ can be chosen to be $P$, such that the following statements hold true:
\begin{enumerate}
 \item\label{it:approximation_improved} If $X=(X_1,\dots,X_n)$ is a tuple of selfadjoint operators in any tracial $W^\ast$-probability space $(\M,\tau)$, then for all $z\in\C^+$ and all $\epsilon>0$
\begin{equation}\label{eq:approximation_improved}
\big| G_{P(X)}(z) - \big[ \bG_{\hat{P}(X)}(\Lambda_\epsilon(z)) \big]_{1,1} \big| \leq \frac{2\epsilon}{\Im(z)^2} \sum^d_{j=1} \|P_j(X)\|_2^2 .
\end{equation}
 \item\label{it:approximation_improved_random_matrices} If $X^{(N)} = (X^{(N)}_1,\dots,X_n^{(N)})$ is a tuple of selfadjoint random matrices belonging to $M_N(L^{\infty -}(\Omega,\bP))$ for any classical probability space $(\Omega,\F,\bP)$ and arbitrary $N\in\N$, then for all $z\in\C^+$ and all $\epsilon>0$
\begin{equation}\label{eq:approximation_improved_random_matrices}
\big| G_{P(X^{(N)})}(z) - \big[ \bG_{\hat{P}(X^{(N)})}(\Lambda_\epsilon(z)) \big]_{1,1} \big| \leq \frac{2\epsilon}{\Im(z)^2} \sum^d_{j=1} \bE\big[\tr_N\big(P_j(X^{(N)})^\ast P_j(X^{(N)})\big)\big].
\end{equation}
\end{enumerate}
\end{theorem}

\begin{proof}
Obviously, with $\rho=(u,Q,v)$ also $\rho_\lambda = (\lambda^{1/2} u, \lambda Q, \lambda^{1/2} v)$ yields a selfadjoint linear representation of $P$ for every $\lambda>0$; thus, since $u\neq 0$ by assumption, we may assume with no loss of generality that $u$ is normalized such that $uu^\ast = 1$. Basic linear algebra tells us that we may find then an orthonormal basis $\{u_1,\dots,u_d\}$ of $\C^d$ with $u_1 = u$. We use these row vectors to define the wanted noncommutative polynomials $P_1,\dots,P_d$ by $P_j := - u_j Q^{-1} v$ for $j=1,\dots,d$; by construction, we clearly have $P_1 = P$.

We shall show that these polynomials $P_1,\dots,P_d$ have the required properties. We will only prove the validity of Item \ref{it:approximation_improved}; the details of the proof of Item \ref{it:approximation_improved_random_matrices} are left to the reader.

Let us take any selfadjoint operators $X_1,\dots,X_n$ in an arbitrary tracial $W^\ast$-probability space $(\M,\tau)$. Further, let us choose $z\in\C^+$ and $\epsilon>0$.

We begin with the observation that the operator $z - u(i\epsilon 1_d - Q(X))^{-1} u^\ast$ is invertible in $\M$ with
\begin{equation}\label{eq:inverse-bound}
\|(z - u(i\epsilon 1_d - Q(X))^{-1} u^\ast)^{-1}\| \leq \frac{1}{\Im(z)}.
\end{equation}
In order to verify this, let us abbreviate $h:=z-u(i\epsilon 1_d - Q(X))^{-1}u^\ast$; we observe that
\begin{align*}
\Im(h) &= \Im(z) - \epsilon u (i\epsilon 1_d + Q(X))^{-1} (i\epsilon 1_d - Q(X))^{-1} u^\ast\\
       &= \Im(z) + \epsilon u (\epsilon^2 1_d + Q(X)^2)^{-1} u^\ast\\
       &\geq \Im(z),
\end{align*}
since $(\epsilon^2 1_d + Q(X)^2)^{-1} \geq 0$. This implies (cf. \cite[Lemma 3.1 (i)]{HT05}) that $h$ is invertible with $\|h^{-1}\| \leq \frac{1}{\Im(z)}$, as desired.

Next, we note that according to the Schur complement formula
$$\big[ \bG_{\hat{P}(X)}(\Lambda_\epsilon(z)) \big]_{1,1} = \tau\big( \big(z-u(i\epsilon 1_d - Q(X))^{-1} u^\ast \big)^{-1} \big).$$
With the help of the resolvent identity, we obtain
$$P(X) - u(i\epsilon 1_d-Q(X))^{-1}u^\ast = -i\epsilon u Q(X)^{-1} (i\epsilon 1_d - Q(X))^{-1} u^\ast$$
and in turn
$$G_{P(X)}(z) - \big[ \bG_{\hat{P}(X)}(\Lambda_\epsilon(z)) \big]_{1,1} = -i\epsilon\, \langle \Psi_1, \Psi_2 \rangle_{L^2(\M,\tau)^d}$$
for the vectors $\Psi_1,\Psi_2$ in $\M^d \subset L^2(\M,\tau)^d$ that are given by
\begin{align*}
\Psi_1 &:= (i\epsilon 1_d-Q(X))^{-1} u^\ast \big(z-u(i\epsilon 1_d-Q(X))^{-1}u^\ast\big)^{-1},\\
\Psi_2 &:= Q(X)^{-1} u^\ast (\overline{z}-P(X))^{-1}.
\end{align*}
Thus, by the Cauchy-Schwarz inequality on $L^2(\M,\tau)^d$, we get that
$$\big|G_{P(X)}(z) - \big[ \bG_{\hat{P}(X)}(\Lambda_\epsilon(z)) \big]_{1,1}\big| \leq \epsilon\, \|\Psi_1\|_{L^2(\M,\tau)^d} \|\Psi_2\|_{L^2(\M,\tau)^d}.$$
One easily sees that
$$\|\Psi_2\|_{L^2(\M,\tau)^d} \leq \frac{1}{\Im(z)}\, \|Q(X)^{-1} u^\ast\|_{L^2(\M,\tau)^d},$$
and similarly, by \eqref{eq:inverse-bound}, we get that
$$\|\Psi_1\|_{L^2(\M,\tau)^d} \leq \frac{1}{\Im(z)}\, \|(i\epsilon 1_d-Q(X))^{-1} u^\ast\|_{L^2(\M,\tau)^d}.$$
By combining these observations, we are led to
\begin{align*}
\lefteqn{\big|G_{P(X)}(z) - \big[ \bG_{\hat{P}(X)}(\Lambda_\epsilon(z)) \big]_{1,1}\big|}\\
& \qquad \leq \frac{\epsilon}{\Im(z)^2} \|Q(X)^{-1} u^\ast\|_{L^2(\M,\tau)^d} \|(i\epsilon 1_d-Q(X))^{-1} u^\ast\|_{L^2(\M,\tau)^d}.
\end{align*}

Furthermore, by the resolvent identity,
\begin{align*}
\lefteqn{\|(i\epsilon 1_d-Q(X))^{-1} u^\ast\|_{L^2(\M,\tau)^d}}\\
&\qquad \leq \|Q(X)^{-1} u^\ast\|_{L^2(\M,\tau)^d} + \epsilon\, \|(i\epsilon 1_d-Q(X))^{-1} Q(X)^{-1} u^\ast\|_{L^2(\M,\tau)^d}\\
&\qquad \leq 2 \|Q(X)^{-1} u^\ast\|_{L^2(\M,\tau)^d}.
\end{align*}
Finally, we involve $1_d = u_1^\ast u_1 + \dots + u_d^\ast u_d$ in order to obtain
$$\|Q(X)^{-1} u^\ast\|_{L^2(\M,\tau)^d}^2 = \sum^d_{j=1} \tau(u Q(X)^{-1} u_j^\ast u_j Q(X)^{-1} u^\ast) = \sum^d_{j=1} \|P_j(X)\|_2^2.$$
Thus, we arrive at \eqref{eq:approximation_improved}, which concludes the proof of Item \ref{it:approximation_improved}.
\end{proof}

\bibliographystyle{amsalpha}
\bibliography{Hoelder_ref}

\end{document}